\documentclass[12pt]{amsart}

\usepackage{amsmath, amssymb, amsthm, mathrsfs}
\usepackage{hyperref}
\usepackage{geometry}
\usepackage{bm}
\newtheorem{theorem}{Theorem}
\newtheorem*{theorem*}{Theorem}
\newtheorem{proposition}[theorem]{Proposition}
\newtheorem{lemma}[theorem]{Lemma}
\newtheorem{corollary}[theorem]{Corollary}

\theoremstyle{definition}

\theoremstyle{remark}
\newtheorem{remark}[theorem]{Remark}

\newcommand{\Z}{\mathbb{Z}}

\author{Pradeep Bisht, Suman Rani, Santanu Tantubay}
\address{Harish-Chandra Research Institute, Prayagraj. A CI of Homi Bhabha National Institute,\\ \small Chhatnag Road, Jhunsi, Prayagraj(Allahabad) 211019 India } 
\email{ pradeepbisht@hri.res.in, pradeepbishthri@gmail.com}
\address{Harish-Chandra Research Institute, Prayagraj. A CI of Homi Bhabha National Institute,\\ \small Chhatnag Road, Jhunsi, Prayagraj(Allahabad) 211019 India } 
\email{sumanrani@hri.res.in}
\address{Shenzhen International Center for Mathematics, Southern University of Sci-
ence and Technology, Shenzhen, Guangdong, China.}
\email{1mathsantanu@gmail.com}
\title{Automorphism groups and derivation  algebras of Hamiltonian Lie algebras}

\date{}

\begin{document}

\begin{abstract}
In this paper, we compute the automorphism group and derivation algebra of the Hamiltonian Lie algebra $\mathcal{H}_{N}$ and its derived subalgebra $\mathcal{H}_{N}'$, where $N$ is an even positive integer. The automorphism groups are shown to be $\mathbf{GSp}_{N}(\mathbb{Z})\ltimes (\mathbb{\mathbb{K}}^{\times})^{N}$ for both Lie algebras and we prove that all derivations are inner for the Hamiltonian Lie algebra, also we compute the full derivation space for the derived subalgebra of Hamiltonian Lie algebra. Finally we compute the second cohomology group of Hamiltonian Lie algebra.
\end{abstract}
\maketitle
\section{Introduction}
\label{sec:introduction}

 Lie algebras of Cartan type over algebraically closed fields of characteristic zero, particularly those arising as derivations of polynomial or Laurent polynomial rings, continue to be a rich source of examples and technique in infinite-dimensional Lie theory. The classical Witt algebra $\mathcal{W}_N$, the special Lie algebra $\mathcal{S}_N$ of divergence zero vector fields and the Hamiltonian Lie algebra $\mathcal{H}_N$ ($N$ even) on torus form an natural hierarchy of such algebras each equipped with a natural $\Z^N$-grading and connected to geometric structures on algebraic tori.  The automorphism groups of Witt algebras have been studied extensively (see, for examples \cite{R86}, \cite{Ba16}, \cite{Ba17}).  \medskip 
 
 The author in \cite{R74} gave sufficient conditions for a Lie algebra to have outer derivations. The derivation algebra of generalized Witt algebra was systematically investigated by Dokovic and Zhao in \cite{DZ98}. The derivation algebra and automorphism groups are determined for higher-rank Virasoro-like algebras in \cite{TX09}. Recently, the derivation algebras for Lie algebras of polynomial vector fields in infinitely many variables have been explicitly described by Bezushchak and Kashuba in \cite{BK25}. 
\medskip

The Hamiltonian Lie algebra $\mathcal{H}_N$ ($N$ even) defined as the subalgebra of Witt algebra $\mathcal{W}_N$ as Hamiltonian vectorfields on $N$-dimensional torus with respect to standard symplectic form occupies a distinctive position within this hierarchy.  It decomposes as 
\begin{equation*}
    \mathcal{H}_N = \mathcal{H}_N' \rtimes \mathfrak{h},\end{equation*}
where $\mathcal{H}_N^\prime=[\mathcal{H}_N,\mathcal{H}_N]$ is the simple [see Proposition \ref{prop:2.4}] derived subalgebra  graded by $\Z^N$ with one dimensional non-zero graded component and $\mathfrak{h}\cong \mathbb{K}^N$ is the Cartan subalgebra of degree zero elements, which tracks the $\Z^N$-gradation of $\mathcal{H}_N$. The Lie bracket is governed by the symplectic bilinear form on $\Z^N$, making $\mathcal{H}_N$ a natural infinite dimensional analogue of finite-dimensional symplectic Lie algebras.

\medskip
Recently weight representation theory of these Lie algebras were studied in \cite{FT25}, \cite{RP25}.

\medskip
While the automorphism groups and derivation algebras have received considerable attentions, the full structure of $\operatorname{Aut}(\mathcal{H}_N),\; \operatorname{Aut}(\mathcal{H}_N^\prime)$ and $\operatorname{Der}(\mathcal{H}_N),\;\operatorname{Der}(\mathcal{H}_N^\prime)$ remained open beyond low rank or special cases. The purpose of this paper is to determine these objects explicitly.

\medskip

 The structure of the paper is as follows.\\
 In Section 2 we first recall the necessary definitions and terminology. Later, we recall and prove lemmas that will be used in sequel. Notably, a generating set for the derived subalgebra of Hamiltonian algebra i.e. $\mathcal{H}_{N}'$ is obtained in Proposition~\ref{prop:2.3}. \\ 
 Section 3 is devoted to obtaining the automorphism groups of the derived subalgebra of Hamiltonian algebra($\mathcal{H}_{N}'$) and the Hamiltonian Lie algebra $\mathcal{H}_{N}$. It is shown in the Lemma~\ref{lem:3.2} that any automorphism of $\mathcal{H}_{N}'$ maps a $\mathbb{Z}^{N}$-graded component of $\mathcal{H}_{N}$ to another graded component. In fact, it is shown that if $\bm{r}$th-graded component is mapped to $\bm{s}$th-graded component, also this association $\bm{r}\mapsto \bm{s}$ is shown to be an odd function, that is, if $\bm{r}\mapsto \bm{s}$, then $-\bm{r}\mapsto -\bm{s}$ . Lemma~\ref{lem:3.1} together with Lemma~\ref{lem:3.2} implies that the association $\bm{r}\to \bm{s}$ actually determines an automorphism of the $\mathbb{Z}$-module $\mathbb{Z}^{N}$, and hence determine a matrix $\bm{Q}\in \mathbf{GL}_{N}(\mathbb{Z})$. The commutator relations in $\mathcal{H}_{N}'$ imply that the matrix $\bm{Q}$ lies in the group $\mathbf{GSp}_{N}(\mathbb{Z})$. Also, the explicit action of automorphism is obtained on the $\mathbb{Z}^{N}$-graded components of $\mathcal{H}_{N}'$ is obtained. The automorphism group of the Lie algebra $\mathcal{H}_{N}'$ described in Theorem~\ref{thm:3.4}. In Lemma~\ref{lem:3.5}, it is shown that any automorphism of the Hamiltonian algebra maps $\mathcal{H}_{N}'$ to $\mathcal{H}_{N}'$ and the Cartan subalgebra $\mathfrak{h}$ to itself. Finally, the automorphism group of $\mathcal{H}_{N}$ is presented in Theorem~\ref{thm:3.6}, which gives us our first theorem: %in Theorem 3.6.
 \begin{theorem*}[A]
    Let $N\geq 2$ be even, then $\operatorname{Aut}(\mathcal{H}_N)\cong \operatorname{Aut} (\mathcal{H}_N^\prime)\cong \mathbf{GSp}_N(\Z)\ltimes (\mathbb{K}^{\times})^{N}$, where $\mathbf{GSp}(\Z)$ is the conformal symplectic group over $\Z$, consisting of matrices $\mathbf{Q}\in \bm{GL}_{N}(\mathbb{Z})$ such that $\bm{Q}^{\top}\bm{J}\bm{Q}=\lambda(Q)\bm{J} \; \text{for some}\; \lambda(Q)\in  \{1,-1\}$ and  $\bm J$ is the standard symplectic matrix.  
 \end{theorem*}
In Section 4, we first prove in Lemma~\ref{lem:4.3} that any graded derivation of $\mathcal{H}_{N}'$ of nonzero degree $\bm{r}$ is inner by obtaining a generating set for $\mathcal{H}_{N}'$ . In Lemma~\ref{lem:4.4}, we describe all graded derivations of degree $\bm{0}$ by doing a series of calculations for constants $c_{\bm{r}}$  , where $ h_{\bm{r}}\to c_{\bm{r}}h_{\bm{r}}$ for all $\bm{r}\in \mathbb{Z}^{N}\setminus\{\bm{0}\}$, under a graded derivation of degree zero, (see ~\eqref{eq:17} and ~\eqref{eq:18}). Derivation algebras of $\mathcal{H}_{N}'$ and $\mathcal{H}_{N}$ is described in Theorem~\ref{thm:4.5} and Theorem~\ref{thm:4.6}. Our second theorem is stated as  

\begin{theorem*}[B]
   Let $N\geq 2$ be even, then $\operatorname{Der}(\mathcal{H}_N)=\operatorname{Der}(\mathcal{H}_N^\prime)\cong  \mathcal{H}_N$. In particular all derivations for Hamiltonian Lie algebras are inner.
\end{theorem*}
     In Section 5, we first show that the second cohomology group $H^{2}(L;\mathbb{K})$ of any $\mathbb{Z}^{N}$-graded Lie algebra is the direct product of its graded components in Lemma~\ref{lem:5.2}. The graded components of $H^{2}(\mathcal{H}_{N};\mathbb{K})$ are explicitly described in Corollary~\ref{cor:5.6} and Lemma~\ref{lem:5.7}. The second cohomology group of $\mathcal{H}_{N}$ with trivial coefficients $H^{2}(\mathcal{H}_{N}; \mathbb{K})$ is characterized in Theorem~\ref{thm:5.8}. We now state the third of our main theorem.
 \begin{theorem*}[C] Let $N\geq 2$ be even, then $H^{2}(\mathcal{H}_{N};\mathbb{K})\cong \Lambda^{2}(\mathbb{K}^{N})\oplus \mathbb{K}^{N}$, where $\Lambda^{2}(\mathbb{K}^{N})$ denotes the second exterior power of $\mathbb{K}^{N}$.
 \end{theorem*}    
 
\section{Preliminaries}
\label{sec:background}
Let $\mathbb{N},\,\mathbb{Z}, \mathbb{Z}_{+},\mathbb{Z}^{*},\,\mathbb{K}$ denote the sets of positive integers, integers, non-negative integers, nonzero integers and an algebraically closed field of characteristic zero. \par 
We fix a positive integer $N\geq 2$. Let $A_{N}= \mathbb{K}[t_{1}^{\pm{1}},\ldots, t_{N}^{\pm {1}}]$ be the commutative algebra of Laurent polynomials over $\mathbb{K}$ and $\mathcal{W}_{N}$ be the Lie algebra of all derivations of $A_{N}$, known as the Witt algebra. Let $d_{i}$ denote the $i$-th degree degree derivation $t_{i}\frac{\partial}{\partial t_{i}}$ for every $1\leq i \leq N$ and $t^{\bm{n}}=t_{1}^{n_{1}}t_{2}^{n_{2}}\ldots t_{N}^{n_{N}}$ for $\bm{n}=(n_{1},\ldots,n_{N})^{\top}\in \mathbb{Z}^{N}$. Also, denote by $e_{i}$ the column vectors in $\mathbb{Z}^{N}$ such that $(e_{i})_{j}=\begin{cases}
    1, \; \text{if} \; i=j \\
    0, \; \text{otherwise}
\end{cases}$.

\par 
Define the bilinear form $\bm{(}\cdot \,, \cdot\bm{)}$ on $\mathbb{K}^{N}$  by $\left(e_{i},e_{j} \right)=\delta_{ij}$, and extending bilinearly to $\mathbb{K}^{N}$. For $u=(u_{1},\ldots, u_{N})^{\top}\in \mathbb{K}^{N}$ and $\bm{r}=(r_{1},\ldots,r_{N})^{\top}\in \mathbb{Z}^{N}$, let $D(u,\bm{r})$ denote the derivation $\sum_{i=1}^{N}u_{i}t^{\bm{r}}d_{i}$. Then the Witt algebra $\mathcal{W}_{N}$ is a $\mathbb{Z}^{N}$-graded Lie algebra with the grading given by
\begin{equation*}
    \mathcal{W}_{N}=\bigoplus_{\bm{r}\in \mathbb{Z}^{N}}(\mathcal{W}_{N})_{\bm{r}},
\end{equation*}  
where $(\mathcal{W}_{N})_{\bm{r}}=\text{span}_{\mathbb{K}}\{ D(u,\bm{r})\mid u\in \mathbb{K}^{N}\}$ for every $\bm{r}\in \mathbb{Z}^{N}$. Then commutator relations on $\mathcal{W}_{N}$ are given by
\begin{equation*}
    [D(u,\bm{r}),D(v,\bm{s})]=D(w,\bm{r}+\bm{s}), \; u,v \in \mathbb{K}^{N},\; \bm{r},\bm{s}\in \mathbb{Z}^{N}.
\end{equation*}
where $w=\bm{(}u,\bm{s}\bm{)}v-\bm{(}v,\bm{r}\bm{)}u$. The abelian subalgebra $\mathfrak{h}=\text{span}_{\mathbb{K}}\{d_{i}\mid 1\leq i \leq N\}=\{D(u,0)\mid u\in \mathbb{K}^{N}\}$ is a Cartan subalgebra of $\mathcal{W}_{N}$. The Witt algebra $\mathcal{W}_{N}$ is naturally isomorphic to the Lie algebra of vector fields on the $N$-dimensional torus $(\mathbb{K}^{\times})^{N}$. Then it has an important subalgebra, the Lie algebra of divergence-zero vector fields $\mathcal{S}_{N}$, given by
\begin{equation*}
    \mathcal{S}_{N}=\bigoplus_{\bm{r}\in \mathbb{Z}^{N}}(\mathcal{S}_{N})_{\bm{r}}
\end{equation*}
where $(\mathcal{S}_{N})_{\bm{r}}=\text{span}_{\mathbb{K}}\{ D(u,\bm{r})\mid u\in \mathbb{K}^{N} \; \text{with} \; \bm{(}u,\bm{r}\bm{)}=0\}$. It is easy to see that $\mathcal{S}_{N}$ contains the Cartan subalgebra $\mathfrak{h}$.\par
 For $N=2m$ for some $ m\in \mathbb{N}$, the Lie algebras $\mathcal{W}_{N}$ as well as $\mathcal{S}_{N}$ has yet another important subalgebra known as the Lie algebra of Hamiltonian vector fields on $N$-dimensional torus $\mathcal{H}_{N}$. For the defining $\mathcal{H}_{N}$, we first require some notations. \par
 Let $\bm{J}=
 \begin{pmatrix}
     \bm{O}_{m}  & \bm{I}_{m} \\
    -\bm{I}_{m}  &\: \bm{O}_{m}
 \end{pmatrix}$,
 where $I_{m}$ denotes identity matrix of order $m$ and $O_{m}$ the zero matrix of order $m$. Then for any $\bm{r}=(r_{1},r_{2},\ldots,r_{m},r_{m+1},\ldots, r_{2m})^{\top}\in \mathbb{Z}^{N}$, we denote by $\overline{\bm{r}}$ the vector $J\bm{r}=(r_{m+1},r_{m+2},\ldots,r_{2m},-r_{1},-r_{2},\ldots,-r_{m})^{\top}$. Note that $\operatorname{char}(\mathbb{K})=0$, therefore $\mathbb{Z}$ is canonically embedded in the field $\mathbb{K}$. Therefore, we  identify $\mathbb{Z}^{N}$ with $\bigoplus_{i=1}^{N}\mathbb{Z}e_{i} $ in $\mathbb{K}^{N}$. Similarly, $\mathbb{Q}$ is canonically embedded in $\mathbb{K}$.  Let us denoted by $h_{\bm{r}}$ the derivations of $A_{N}$ of the form $D(\overline{\bm{r}},\bm{r})$. Then the Hamiltonian Lie algebra is defined as \begin{equation*}
     \mathcal{H}_{N}=\text{span}_{\mathbb{K}}\{h_{\bm{r}}\mid \bm{r}\in \mathbb{Z}^{N}\}\rtimes \mathfrak{h}
 \end{equation*}
 Note that the Hamiltonian algebra is not simple, whereas its derived subalgebra $\mathcal{H}_{N}'=[\mathcal{H}_{N},\mathcal{H}_{N}]$ is. Also, we note that \begin{equation*}
     \mathcal{H}_{N}= \bigoplus_{\bm{r}\in \mathbb{Z}^{N}}(\mathcal{H}_{N})_{\bm{r}}=\mathcal{H}_{N}'\rtimes \mathfrak{h}, \;\; \;  \mathcal{H}_{N}'=\bigoplus_{\bm{r}\in \mathbb{Z}^{N}\setminus\{\bm{0}\}}(\mathcal{H}_{N})_{\bm{r}}=\bigoplus_{\bm{r}\in \mathbb{Z}^{N}}(\mathcal{H}_{N}')_{\bm{r}},
 \end{equation*}
 where $(\mathcal{H}_{N})_{\bm{r}}=(\mathcal{H}_{N}')_{\bm{r}}=\mathbb{K}h_{\bm{r}}$ for $\bm{r}\neq 0$, $(\mathcal{H}_{N})_{\bm{0}}=\mathfrak{h}$ and $(\mathcal{H}_{N}')_{\bm{0}}=0$. For the later use, we record the following; 
 \begin{equation*}
     [h_{\bm{r}},h_{\bm{s}}]=(\overline{\bm{r}},\bm{s})h_{\bm{r}+\bm{s}},
 \end{equation*}
 where $\bm{r},\bm{s}\in \mathbb{Z}^{N}$ and $(\overline{\bm{r}},\bm{s})=\sum_{i=1}^{m}r_{m+i}s_{i}-\sum_{i=1}^{m}r_{i}s_{m+i}$.
 \par 
 Now we recall the following groups which are important for later purpose. \\
 For $N=2m$ with $m\in \mathbb{N}$, let $\bm{GL}_{N}(\mathbb{Z})$ denotes the group of all $ N\times N$ integer matrices with determinant $\pm 1$. The symplectic group over the integers is defined as 
 \begin{equation*}
     \bm{Sp}_{N}(\mathbb{Z})=\{\bm{Q} \in \bm{GL}_{N}(\mathbb{Z}) \mid \bm{Q}^{\top}\bm{J}\bm{Q}=\bm{J} \},
 \end{equation*}
 where $\bm{J}=
 \begin{pmatrix}
     \bm{O} & \bm{I}_{m} \\
     -\bm{I}_{m} & \bm{O}
 \end{pmatrix}.$
 \par 
 The conformal symplectic group over $\mathbb{Z}$ is defined as
 \begin{equation*}
     \bm{GSp}_{N}(\mathbb{Z})=\{\bm{Q}\in \bm{GL}_{N}(\mathbb{Z}) \mid   \bm{Q}^{\top}\bm{J}\bm{Q}=\lambda(Q)\bm{J} \; \text{for some}\; \lambda(Q)\in \mathbb{Z}^{\times}=\{-1,1\}\}.
 \end{equation*}
 See \cite{MT11}, where they have defined the conformal symplectic group of any field $k$. \\
 We note that $\bm{GSp}_{N}(\mathbb{Z})\cong \bm{Sp}_{N}(\mathbb{Z})\rtimes \mathbb{Z}/2\mathbb{Z} $. Since $\lambda : \bm{GSp}_{N}(\mathbb{Z})\to \mathbb{Z}/2\mathbb{Z}$ defined by
 \begin{equation*}
     \lambda(\bm{A})=
     \begin{cases}
         0 & \text{if} \; \bm{A} \in \bm{Sp}_{N}(\mathbb{Z}) \\
         1 & \text{if} \; \bm{A} \notin \bm{Sp}_{N}(\mathbb{Z})
     \end{cases}
 \end{equation*}
 is a group homomorphism. The following sequence is split exact.  
 \begin{equation*}
     1 \longrightarrow \bm{Sp}_{N}(\mathbb{Z}) \xrightarrow{i} \bm{GSp}_{N}(\mathbb{Z})\xrightarrow{\lambda} \mathbb{Z}/2\mathbb{Z}  \longrightarrow 1
 \end{equation*}
 hence, we have $\bm{GSp_{N}}(\mathbb{Z})\cong \bm{Sp}_{N}(\mathbb{Z})\rtimes \mathbb{Z}/2\mathbb{Z}$. \\  
  Recall that $\bm{r}=(r_{1},\ldots,r_{N})^{\top}\in \mathbb{Z}^{N} $ is said to be primitive if $\gcd(r_{1},\ldots,r_{N})=1$. The following Lemma is due to Yves Benoist.
  \begin{lemma}\label{lem:2.1}{[[Ben25], Lemma 5.3.]} The group $\bm{Sp}_{N}(\mathbb{Z})$ acts transitively on the set of primitive vectors in $\mathbb{Z}^{N}$.
  \end{lemma}
Now we record the following facts in terms of lemma, which will be employed later in the paper.
\begin{lemma}\label{lem:2.2} Let $S\subset \mathcal{H}_{N}$ and $\langle S\rangle$ denotes the Lie subalgebra of $\mathcal{H}_{N}$ generated by $S$, then \begin{itemize}
    \item[(i)] If the $h_{\bm{r}},h_{\bm{s}}\in S$, then $h_{\bm{r}+\bm{s}}\in \langle S \rangle $ if $(\overline{\bm{r}},\bm{s})\neq 0$. 
    \item[(ii)] If $h_{\bm{r}}\in S $ and $h_{\bm{s}}\in \langle S \rangle $, then $h_{\bm{r}+\bm{s}}\in \langle S\rangle $ if $(\overline{\bm{r}},\bm{s})\neq 0$.
    \item[(iii)] If $h_{\bm{r}}, h_{\bm{s}}\in \langle S\rangle $, then $h_{\bm{r}+\bm{s}}\in \langle S\rangle $ if $(\overline{\bm{r}},\bm{s})\neq 0$.
\end{itemize}
\end{lemma}
\begin{proof}
    Follows from Lie brackets \begin{equation*}
        [h_{\bm{r}},h_{\bm{s}}]=(\overline{\bm{r}},\bm{s})h_{\bm{r}+\bm{s}},
    \end{equation*}
    and the fact that the nonzero graded components of $\mathcal{H}_{N}'$ are one dimensional.
\end{proof}
For any $\{0\}\neq S $, we define $L_{S}=\{\bm{r}\in \mathbb{Z}^{N} \mid h_{\bm{r}}\in \langle S \rangle \}$, then by Lemma 2.2. $L_{S}$ has the property that if $\bm{r},\bm{s}\in L_{S}$ and $\left(\overline{\bm{r}},\bm{s}\right)\neq 0$, then $\bm{r}+\bm{s}\in L_{S}$. By the definition, $\bm{0}\in L_{S}$.

\medskip

 We prove the following important technical Proposition, which will be very helpful while computing the automorphism group.
 
\begin{proposition}\label{prop:2.3}
    If $ S=\{\pm e_{i}, \pm e_{j}\pm e_{k} \mid 1\leq i \leq N, 1\leq j < k \leq N\}$, then $L_{S}=\mathbb{Z}^{N}$.
\end{proposition}
\begin{proof}
    \textbf{Claim 1:}. Every vector of the form $\sum_{i=1}^{N}\epsilon_{i}e_{i}$ with $\epsilon_{i}\in \{\pm 1\}$ lies in $L_{S}$. 
    %For any $\epsilon_{i}\in \{-1,1\}$, $\sum_{i=1}^{N}\epsilon_{i}e_{i}\in L_{S} $. 

    \medskip 

     For $m=1$, we are already done. \\
      For $m=2$, if $\bm{r}_{1}=\epsilon_{1}e_{1}+\epsilon_{2}e_{2}$, then $\left(\overline{\bm{r}_{1}}, \epsilon_{3}e_{3}\right)= -\epsilon_{1}\epsilon_{3}\neq 0$, by Lemma 2.2, we have $\bm{r}_{2}=\epsilon_{1}e_{1}+\epsilon_{2}e_{2}+\epsilon_{3}e_{3}\in L_{S}$. Note that $\left(\overline{\bm{r}_{3}},\epsilon_{4}e_{4}\right)=-\epsilon_{2}\epsilon_{4}\neq 0$, therefore , we have $\epsilon_{1}e_{1}+\epsilon_{2}e_{2}+\epsilon_{3}e_{3}+\epsilon_{4}e_{4}\in L_{S}$. \\
    
    For $m>2$, if $\bm{r}_{1}=\epsilon_{1}e_{1}+\epsilon_{2}e_{2}$ and $\bm{r}_{2}= \epsilon_{3}e_{3}+\epsilon_{m+1}e_{m+1} \in L_{S} $ then we have $(\overline{\bm{r}_{1}},\bm{r}_{2})=-\epsilon_{1}\epsilon_{m+1}\neq 0$,  and hence it follows by Lemma~\ref{lem:2.2} that $\bm{r}_{1}+\bm{r}_{2}=\epsilon_{1}e_{1}+\epsilon_{2}e_{2}+\epsilon_{3}e_{3}+\epsilon_{m+1}e_{m+1}\in L_{S}$. 

\medskip
    
    Now if we choose $\bm{r}_{3}=\epsilon_{4}e_{4}+\epsilon_{m+2}e_{m+2}$, then we have $(\overline{\bm{r}_{1}+\bm{r}_{2}},\bm{r}_{3}) = -\epsilon_{2}e_{m+2}\neq  0$, and again by Lemma~\ref{lem:2.2} we see that $\bm{r}_{1}+\bm{r}_{2}+\bm{r}_{3}= \sum_{i=1}^{4}\epsilon_{i}e_{i} + \sum_{i=1}^{2}\epsilon_{m+i}e_{m+i} \in L_{S} $.
    
    \medskip
    
    Assume as the induction hypothesis that for any  $k\in \{3,\ldots, m-1\}$, we have $\bm{r}=\sum_{i=1}^{k}\epsilon_{i}e_{i} + \sum_{i=1}^{k-2}\epsilon_{m+i}e_{m+i} \in L_{S} $. If $\bm{s}= \epsilon_{k+1}e_{k+1}+ \epsilon_{m+k-1}e_{m+k-1}$, then $\left( \overline{\bm{r}},\bm{s}\right) = -e_{k-1}e_{m+k-1}\neq 0$ and again by Lemma~\ref{lem:2.2} gives that $\bm{r}+\bm{s}= \sum_{i=1}^{k+1}\epsilon_{i}e_{i}+\sum_{i=1}^{(k+1)-2}\epsilon_{i}e_{i} \in L_{S}$. It follows by induction that \[\sum_{i=1}^{m}\epsilon_{i}e_{i}+\sum_{i=1}^{m-2}\epsilon_{m+i}e_{m+i} =\sum_{i=1}^{2m-2}\epsilon_{i}e_{i}\in L_{S}.\] \\
    Now if $\bm{r}'=\sum_{i=1}^{2m-2}\epsilon_{i}e_{i}$ and $\bm{s}'=\epsilon_{2m-1}e_{2m-1}$, then  $\left(\overline{\bm{r}'},\bm{s}' \right)=-\epsilon_{m-1}\epsilon_{2m-1}\neq 0$  and hence Lemma~\ref{lem:2.2} implies that $\bm{r}'+\bm{s}'=\sum_{i=1}^{2m-1}\epsilon_{i}e_{i}\in L_{S}$. Similarly $\bm{t}=\epsilon_{2m}e_{2m}\in L_{S}$ and $ \bm{r}'+\bm{s}'\in L_{S}$ and $(\overline{\bm{r}'+\bm{s}'},\bm{t})=-\epsilon_{m}\epsilon_{2m}\neq 0$ implies $ \bm{r}' +\bm{s}'+\bm{t}=\sum_{i=1}^{2m}\epsilon_{i}e_{i}\in L_{S}$. 

    \medskip
    
    \textbf{Claim 2: } We have $(\mathbb{Z}^{*})^{N}\subseteq L_{S}$, 
    in other words for any $\epsilon_{i}\in \{-1,1\}$ we will show that $ \sum_{i=1}^{N}k_{i}\epsilon_{i}e_{i}$ lies in $L_S$ for every $(k_{1},\ldots,k_{N})^{\top}\in (\mathbb{Z}_{+})^{N}$.

    \medskip
    
    It is already shown in Claim 1 that $\sum_{i=1}^{N}\epsilon_{i}e_{i}\in L_{S}$. Now we consider $\bm{k}_{1}=\sum_{i=1}^{N}\epsilon_{i}e_{i}$, then we have  $\bm{k_1}, \epsilon_{1}e_{1}\in L_{S}$ and $\left(\overline{\bm{k}_{1}}, \epsilon_{1}e_{1}\right)\neq 0$ implies \[\bm{k}_{2}= 2\epsilon_{1}e_{1} +\sum_{i=2}^{N}\epsilon_{i}e_{i}\in L_{S}.\] 
    
     Using induction hypothesis we assume that for  ${i}\in \{2,\ldots,r_{1}-1\}$, we have 

\begin{equation*}
 \bm{k}_{i} =i\epsilon_{1}e_{1} + \sum_{i=2}^{N}\epsilon_{i}e_{i}  \in L_{S}  \end{equation*}
 and we show that this is also true for $i+1$. 

\medskip
 
        Observe that $\left(\overline{\bm{k}_{i}}, \epsilon_{1}e_{1} \right)\neq 0$, $\epsilon_{1}e_{1},\bm{k}_{i}\in L_{S}$, therefore Lemma~\ref{lem:2.2} along with the induction hypothesis gives $ k_{i+1}= (i+1)\epsilon_{1}e_{1} +\sum_{i=2}^{N}\epsilon_{i}e_{i}\in L_{S}$. Now by induction we have \[\bm{k}_{r_{1}}=r_{1}\epsilon_{1}e_{1} + \sum_{i=2}^{N}\epsilon_{i}e_{i}\in L_{S}.\]
For $m=1$, We have $\left(\overline{\bm{k}_{r_{1}}}, \epsilon_{2}e_{2} \right)= -r_{1}\epsilon_{1}\epsilon_{2}\neq 0$ and Lemma~\ref{lem:2.2} implies $\bm{k}_{r_{1},2}= r_{1}\epsilon_{1}e_{1}+ 2\epsilon_{2}e_{2}\in L_{S}$. Inductively, we deduce that $\bm{k}_{r_{1},r_{2}}=r_{1}\epsilon_{1}e_{1}+r_{2}\epsilon_{2}e_{2}\in L_{S}$. \\
    For $m\geq 2$, we note that $(\overline{\bm{k}_{r_{1}}},\epsilon_{2}e_{2})=\epsilon_{2}\epsilon_{m+2}\neq 0 $, therefore Lemma~\ref{lem:2.2} gives that $\bm{k}_{r_{1},2}= r_{1}\epsilon_{1}e_{1} + 2\epsilon_{2}e_{2} + \sum_{i=3}^{N}\epsilon_{i}e_{i}\in L_{S}$. Repeating this process as above, we obtain
        \begin{equation*}
            \bm{k}_{r_{1},r_{2}}= r_{1}\epsilon_{1}e_{1}+ r_{2}\epsilon_{2}e_{2}+\sum_{i=3}^{N}\epsilon_{i}e_{i}\in L_{S}.
        \end{equation*}
      For $m=2$, this means that $ r_{1}\epsilon_{1}e_{1}+r_{2}\epsilon_{2}e_{2}+\sum_{i=m+1}^{N}\epsilon_{i}e_{i}$. We prove the same for any $m>2$.\\
        For $m>2$, we assume  by induction hypothesis that for any $j \in \{2,\ldots,m-1\}$, 
        \[\bm{k}_{r_{1}.r_{2},\ldots,r_{j}}= \sum_{i=1}^{j}r_{i}\epsilon_{i}e_{i}+ \sum_{i=j+1}^{N}\epsilon_{i}e_{i} \in L_{S}.\] Since $\left(k_{r_{1},\ldots,r_{j}},\epsilon_{j+1}e_{j+1} \right) = \epsilon_{j+1}\epsilon_{m+j+1}\neq 0$ and the above method can be repeatedly used to obtain $\bm{k}_{r_{1},\ldots,r_{j+1}}= \sum_{i=1}^{j+1}r_{i}\epsilon_{i}e_{i} + \sum_{i=j+2}^{N}\epsilon_{i}e_{i} \in L_{S}$. By induction, we have \begin{equation*}
            \bm{k}_{r_{1},\ldots, r_{m}}=\sum_{i=1}^{m}r_{i}\epsilon_{i}e_{i} + \sum_{j=m+1}^{N}\epsilon_{i}e_{i}\in L_{S}.
        \end{equation*}

       In each case, we see that  $\left(\overline{\bm{k}_{r_{1},\ldots, r_{m}}}, \epsilon_{m+1}e_{m+1} \right)=-r_{1}\epsilon_{1}\epsilon_{m+1}\neq 0$, therefore Lemma~\ref{lem:2.2} implies $ \sum_{i=1}^{m}r_{i}\epsilon_{i}e_{i}+ 2 \epsilon_{m+1}e_{m+1} + \sum_{i=m+2}^{N}\epsilon_{i}e_{i} \in L_{S}$ and repeating this method as above yields $\sum_{i=1}^{m+1}r_{i}\epsilon_{i}e_{i} + \sum_{i=m+2}^{N}\epsilon_{i}e_{i} \in L_{S}$. Finally, as above we use induction to obtain \begin{equation*}
            \bm{k}_{r_{1},r_{2},\ldots ,r_{N}}= \sum_{i=1}^{N}r_{i}\epsilon_{i}e_{i} \in L_{S}.
        \end{equation*}
        \textbf{Claim 3:} We have $\mathbb{Z}^{N}\setminus\{\bm{0}\}\subseteq L_{S}$.\\
        Let $\bm{r}\in \mathbb{Z}^{N}\setminus\{\bm{0}\}$, then we can choose $\bm{s}\in (\mathbb{Z}^{*})^{N}$ with $\bm{r}+\bm{s}\in (\mathbb{Z}^{*})^{N}$ and $(\overline{\bm{r}},\bm{s})\neq 0$. By Claim 2, we have $\bm{r}+\bm{s} ,-\bm{s} \in L_{S}$ and since $\left(\overline{\bm{r}+\bm{s}}, -\bm{s}\right)=-\left( \overline{\bm{r}},\bm{s}\right)\neq 0$, Lemma 2.2 yields $\bm{r}+\bm{s}-\bm{s}=\bm{r}\in L_{S}$.
\end{proof}
 The following proposition is due to Prof. Rao (private communication), as it seems to be unpublished, we provide a proof of it.
\begin{proposition}\label{prop:2.4} \cite{Rao}.
The derived subalgebra of Hamiltonian Lie algebra $\mathcal{H}_N$ is a simple Lie algebra.
\end{proposition}
\begin{proof}
    Let $I$ be a nonzero ideal of $\mathcal{H}_{N}'$, then there exists a nonzero element  $X=\sum_{i=1}^{n}a_{i}h_{\bm{r}_{i}}\in I$, where $a_{i}\in \mathbb{K}\setminus\{0\}$ and $\bm{r}_{i}\in \mathbb{Z}^{N}\setminus\{\bm{0}\}$. There are two possible cases, namely, $\bm{r}_{j}\notin \mathbb{Q}\bm{r}_{1}$ for at least one $j \in \{2,\ldots, n\}$ or $\bm{r}_{j}\in \mathbb{Q}\bm{r}_{1}$ for every $j\in \{2,\ldots,n\}$. 

    \medskip 
    
    \textbf{Case 1:} If $\bm{r}_{j}\notin \mathbb{Q}\bm{r}_{1}$ for some $i\in \{2,\ldots,n\}$, then we choose $\bm{s}\in \mathbb{Z}^{N}\setminus\{\bm{0}\}$ such that $(\overline{\bm{s}},\bm{r}_{1})\neq 0$ and $(\overline{\bm{s}},\bm{r}_{i})=0$. Consider $[h_{\bm{s}},X]= \sum_{i=1}^{n}a_{i}(\overline{\bm{s}},\bm{r}_{i})h_{\bm{s}+\bm{r}_{i}}\in I$. Therefore, we obtain a nonzero element in $I$ which can be written as the linear combination of at most $n-1$ elements of the form $h_{\bm{s}+\bm{r}}$. 

    \medskip 
    
    \textbf{Case 2:} If $\bm{r}_{j}\in \mathbb{Q}\bm{r}_{1}$ for every $j\in \{2,\ldots,n\}$, then we can find $\bm{s}\in \mathbb{Z}^{N}\setminus\{\bm{0}\}$ such $(\overline{\bm{s}},\bm{r}_{1})\neq 0$. We have $[h_{\bm{s}},X]=\sum_{i=1}^{N}a_{i}(\overline{\bm{s}},\bm{r}_{i})h_{\bm{s}+\bm{r}_{i}}$. Note that $\bm{s}+\bm{r}_{i}\notin \mathbb{Q}(\bm{s}+\bm{r}_{1})$, therefore, by the same argument as used in Case 1, we obtain a nonzero element in $I$, which can be written as the linear combination of at most $n-1$ elements of the form $h_{\bm{r}}$. 
\medskip 
    
    In each case, we obtain a nonzero element of the form $\sum_{i=1}^{n-1}b_{i}h_{\bm{s}_{i}}$ with $b_{1}\neq 0$. If $b_{i}=0$ for all $i \in \{2,\ldots, n-1\}$, then we get that $h_{\bm{s}_{1}}\in I$. Otherwise, the above methods can be repeated several times to obtain an element of the form $h_{\bm{r}}\in  I$, as in each step, we obtain a nonzero elements as a finite series with one less term.  Thus we obtain $h_{\bm{r}}\in I$.

    \medskip
    
    Next we show that $h_{\bm{s}}\in I$ for all $\bm{s}\in \mathbb{Z}^{N}\setminus\{\bm{0}\}$. \\
    If $(\overline{\bm{s}},\bm{r})\neq 0$, then we have $[h_{\bm{s}-\bm{r}},h_{\bm{r}}]= (\overline{\bm{r}},\bm{s})h_{\bm{s}}\in I$.\\
    If $(\overline{\bm{s}},\bm{r})=0$, then we choose $\bm{t}\in \mathbb{Z}^{N}\setminus\{\bm{0}\}$ with $(\overline{\bm{t}},\bm{s})\neq 0$ and $(\overline{\bm{t}},\bm{s})\neq 0
    $. Then $[h_{\bm{t-\bm{r}}},h_{\bm{r}}]= (\overline{\bm{t}},\bm{r})h_{\bm{t}}\in I$ implies $h_{\bm{t}}\in I$. Consider $[h_{\bm{s}-\bm{t}},h_{\bm{t}}]=(\overline{\bm{s}},\bm{t})h_{\bm{s}}\in I$, and therefore, we have $h_{\bm{s}}\in I$ for all $s\in \Z^N\setminus\{0\}$. \\
    It follows that $I=\mathcal{H}_{N}'$.
\end{proof}
\begin{lemma}\label{lem:2.5}
    If f and g are linear functionals on a finite dimensional space over field $\mathbb{K}$ such that $\operatorname{ker}(f)=\operatorname{ker}(g)$, then there exists a nonzero $c\in \mathbb{K}$ such that $f=cg$.
\end{lemma}
 
\section{Automorphism groups of Hamiltonian Lie algberas}
\label{sec:main-section}
We denote by $\operatorname{Aut}(\mathcal{H}_{N}')$ and $\operatorname{Aut}(\mathcal{H}_{N})$ to denote the automorphism groups of $\mathcal{H}_{N}'$ and $\mathcal{H}_{N}$, respectively. For $N=2$, it is  known that $\mathcal{H}_{2}=\mathcal{S}_{2}$, and thus the groups $\operatorname{Aut}(\mathcal{H}_{2}')$ and $\operatorname{Aut}(\mathcal{H}_{2})$ are obtained in \cite{TX09}. We now proceed to find  $\operatorname{Aut}(\mathcal{H}_{N}')$ and $\operatorname{Aut}(\mathcal{H}_{N})$.
\begin{lemma}\label{lem:3.1}
    Suppose that $\sigma \in \operatorname{Aut}(\mathcal{H}_{N}')$ or $\operatorname{Aut}(\mathcal{H}_{N})$ be such that $\sigma(\mathcal{H}_{N}')=\mathcal{H}_{N}'$. If for every $\bm{r}\in \mathbb{Z}^{N}\setminus\{\bm{0}\}$ there is $\bm{s}\in \mathbb{Z}^{N}\setminus\{\bm{0}\}$ such that $\sigma(\mathcal{H}_{\bm{r}}')= \mathcal{H}_{\bm{s}}'$ and $\sigma(\mathcal{H}_{-\bm{r}}')=\mathcal{H}_{-\bm{s}}'$, then there exist $\lambda \in (\mathbb{K}^{\times})^{N}, \; \bm{Q}\in \bm{GSp}_{N}(\mathbb{Z})$ such that
    \begin{equation*}
        \sigma \left(D(\overline{\bm{r}},\bm{r})\right) = \begin{cases} \lambda^{\bm{r}}D\left(\bm{Q}^{-\top}\overline{\bm{r}}, \bm{Q}\bm{r}\right), & \text{if} \; \;  \bm{Q}^{\top}\bm{J}\bm{Q}= \bm{J}  \\
        (-1)^{|\bm{r}|}\lambda^{\bm{r}} D\left(\bm{Q}^{-\top}\overline{\bm{r}}, \bm{Q}\bm{r}\right),& \text{if}\; \; \bm{Q}^{\top}\bm{J}\bm{Q}= -\bm{J}  \end{cases}
    \end{equation*}\
    where $\bm{r}=(r_{1},\ldots,r_{N})^{\top}\in \mathbb{Z}^{N}\setminus\{\bm{0}\} $ with $|\bm{r}|=\sum_{i=1}^{N}r_{i}$.
\end{lemma}
\begin{proof}
    For every $\bm{r}\in \mathbb{Z}^{N}\setminus\{\bm{0}\}$, let $\bm{s}\in \mathbb{Z}^{N}\setminus\{\bm{0}\}$ satisfy the above hypothesis and define $\bm{s}=f(\bm{r})$. Since $\sigma$ is an automorphism of $\mathcal{H}_{N}'$ and each non-zero graded component of $\mathcal{H}_{N}'$ is one dimensional, the element $\bm{s}=f(\bm{r})$ is uniquely determined for every $\bm{r}\neq\bm{0}$. Hence $f:\mathbb{Z}^{N}\to\mathbb{Z}^{N}$ is bijective.

    \medskip 
    
    Now we will show that $f$ is a $\mathbb{Z}$-module homomorphism. \\
    For $\bm{r},\bm{s}\in \mathbb{Z}^{N}\setminus\{\bm{0}\}$ with $(\overline{\bm{r}},\bm{s})\neq 0$ we have
\[
\left[D(\overline{\bm{r}},\bm{r}), D(\overline{\bm{s}},\bm{s})\right]
= (\overline{\bm{r}},\bm{s})D\big(\overline{\bm{r}}+\overline{\bm{s}},\bm{r}+\bm{s}\big).
\]
Since we know each non-zero graded component is one-dimensional, it follows that
\[
[\mathcal{H}_{\bm{r}}',\mathcal{H}_{\bm{s}}'] = \mathcal{H}_{\bm{r}+\bm{s}}',
\qquad\text{whenever }(\overline{\bm{r}},\bm{s})\neq 0.
\]
Applying $\sigma$ to this equality yields
\[
[\mathcal{H}_{f(\bm{r})}',\mathcal{H}_{f(\bm{s})}']
= \mathcal{H}_{f(\bm{r}+\bm{s})}'.
\]
But the left-hand side equals $\mathcal{H}_{f(\bm{r})+f(\bm{s})}'$, hence
\[
f(\bm{r}+\bm{s}) = f(\bm{r}) + f(\bm{s}),
\qquad\text{whenever }(\overline{\bm{r}},\bm{s})\neq 0. \tag{1}\label{eq:1}
\]

\medskip

Now we extend~\eqref{eq:1} to the remaining case $(\overline{\bm{r}},\bm{s})=0$. First we prove negation compatibility.

Fix $\bm{r}\in\mathbb{Z}^{N}\setminus\{\bm{0}\}$ and choose $\bm{s}\in\mathbb{Z}^{N}\setminus\{\bm{0}\}$ with $(\overline{\bm{r}},\bm{s})\neq 0$. Then
\[
\bm{r} = (\bm{r}+\bm{s}) + (-\bm{s}),
\]
and applying~\eqref{eq:1} to the pair $(\bm{r}+\bm{s}, -\bm{s})$ (valid since $(\overline{\bm{r}+\bm{s}}, -\bm{s})=-(\overline{\bm{r}},\bm{s})\neq 0$) gives
\[
f(\bm{r}) = f(\bm{r}+\bm{s}) + f(-\bm{s}).
\]
Since $f(\bm{r}+\bm{s}) = f(\bm{r}) + f(\bm{s})$ by~\eqref{eq:1}, we obtain
\[
f(-\bm{s}) = -f(\bm{s}) \quad\text{for all }\bm{s}\neq 0. \tag{2}\label{eq:2}
\]
Now let $\bm{r},\bm{s}\in \mathbb{Z}^{N}\setminus\{\bm{0}\}$ be such that $\bm{r}+\bm{s}\neq 0$ and $(\overline{\bm{r}},\bm{s})=0$. Choose $\bm{s}'\in\mathbb{Z}^{N}$ with
\[
(\overline{\bm{r}},\bm{s}')\neq 0,\quad
(\overline{\bm{s}},\bm{s}')\neq 0,\quad
(\overline{\bm{r}}+\overline{\bm{s}},\bm{s}')\neq 0.
\]
Then
\begin{align}
f(\bm{r}+\bm{s})
&= f\big((\bm{r}+\bm{s}')+(-\bm{s}'+\bm{s})\big)\notag \\
&= f(\bm{r}+\bm{s}') + f(-\bm{s}'+\bm{s}) \notag\\
&= \big(f(\bm{r}) + f(\bm{s}')\big) + \big(-f(\bm{s}') + f(\bm{s})\big) \notag\\
&= f(\bm{r}) + f(\bm{s}).\tag{3}
\end{align}
Thus $f(\bm{r}+\bm{s}) = f(\bm{r}) + f(\bm{s})$ holds for all $\bm{r},\bm{s}\in\mathbb{Z}^{N}$. By induction,
\[
f(n\bm{r}) = nf(\bm{r}) \qquad \forall\, n\in\mathbb{Z},\ \bm{r}\in\mathbb{Z}^{N}. \tag{4}\label{eq:4}
\]
Hence  $f: \mathbb{Z}^{N}\to \mathbb{Z}^{N}$ is an isomorphism of the $\mathbb{Z}$-module $\mathbb{Z}^{N}$. Then the coefficients of $e_{j}$'s in the expression of $f(e_{i})$ as a linear combination of $e_{j}$'s determine an invertible matrix $\bm{Q}$ with entries in $\mathbb{Z}$ having determinant $\pm{1}$, i.e. $\bm{Q}\in \mathbf{GL}_{N}(\mathbb{Z})$ such that 
\[
f(\bm{r}) = \bm{Q}\bm{r}\quad\forall\,\bm{r}\in\mathbb{Z}^{N}. \tag{5}\label{eq:5}
\]
Note that each homogeneous component of $\mathcal{H}_{N}'$ of degree $\bm{r}\in \mathbb{Z}^{N}\setminus\{0\}$ is one dimensional. It follows from the definition of $f:\mathbb{Z}^{N}\to \mathbb{Z}^{N}$ and \eqref{eq:5} there exists $c_{\bm{r}}\in \mathbb{K}^{\times}$  such that \begin{equation*}
    \sigma(h_{\bm{r}})=\sigma(D(\overline{\bm{r}},\bm{r}))=c_{\bm{r}}D(\overline{\bm{Q}\bm{r}},\bm{Q}\bm{r})=c_{\bm{r}}h_{\bm{Q}\bm{r}}.
\end{equation*} 
We now show that $\bm{Q}\in\bm{GSp}_N(\mathbb{Z})$. For any $\bm{r},\bm{s}\in\mathbb{Z}^{N}\setminus\{\bm{0}\}$ we have
\[
\left[D(\overline{\bm{r}},\bm{r}),D(\overline{\bm{s}},\bm{s})\right]
= (\overline{\bm{r}},\bm{s})D(\overline{\bm{r}}+\overline{\bm{s}},\bm{r}+\bm{s}).
\]
Applying $\sigma$ and using~\eqref{eq:5}, we obtain scalars $c_{\bm{r}},c_{\bm{s}},c_{\bm{r}+\bm{s}}\in\mathbb{K}^\times$ such that
\[
c_{\bm{r}}c_{\bm{s}}\left(\overline{f(\bm{r})},f(\bm{s})\right)D\left(\overline{f(\bm{r})}+\overline{f(\bm{s})},f(\bm{r})+f(\bm{s})\right)
= (\overline{\bm{r}},\bm{s})c_{\bm{r}+\bm{s}}D\left(\overline{f(\bm{r})}+\overline{f(\bm{s})},f(\bm{r})+f(\bm{s})\right),
\]
hence
\[
c_{\bm{r}}c_{\bm{s}}\left(\overline{f(\bm{r})},f(\bm{s})\right)
= (\overline{\bm{r}},\bm{s})c_{\bm{r}+\bm{s}}. \tag{6}\label{eq:6}
\]
In particular,
\[
(\overline{\bm{r}},\bm{s}) = 0
\quad\Longleftrightarrow\quad
\left(\overline{f(\bm{r})},f(\bm{s})\right) = 0.
\]
Notice that $(,)_{f}$ defined by $(\bm{r},\bm{s})_{f}=\left(\overline{f(\bm{r})},f(\bm{s})\right)$ is again a bilinear form which by~\eqref{eq:6} equals $\displaystyle\frac{c_{\bm{r}+\bm{s}} }{c_{\bm{r}}c_{\bm{s}}}\left(\overline{\bm{r}},\bm{s}\right)$. \\
We show that $c_{\bm{r}+\bm{s}}/c_{\bm{r}}c_{\bm{s}}=c\in \mathbb{K}^{\times}$ for every $\bm{r},\bm{s}\in \mathbb{Z}^{N}\setminus\{\bm{0}\}$. For simplicity, we denote by $\omega(\bm{r},\bm{s})$ the expression $\left(\overline{\bm{r}},\bm{s}\right)$ for any $\bm{r},\bm{s}\in \mathbb{Z}^{N}\setminus\{\bm{0}\}$. \\
Note that for $\bm{r}\in \mathbb{Z}^{N}\setminus\{\bm{0}\}$ the linear functionals $\omega(\bm{r},-)$ and $(\bm{r},-)_{f}$ have the same kernel. By Lemma~\ref{lem:2.5}, the exists $b_{\bm{r}}\in \mathbb{K}^{\times}$ such that $(\overline{\bm{r}},\bm{s})=b_{\bm{r}}(\bm{r},\bm{s})_{f}=b_{\bm{r}}(\overline{f(\bm{r})},f(\bm{s}))$ for every $\bm{s}\in \mathbb{Z}^{N}$. \\
Again, for any $\bm{s}\in \mathbb{Z}^{N}\setminus\{\bm{0}\}$, the linear functionals $\omega(-,\bm{s})$ and $(-,\bm{s})_{f}$ have the same kernel, Consequently, by Lemma~\ref{lem:2.5}, there is $d_{\bm{s}}\in \mathbb{K}^{\times}$ such that  $(\overline{\bm{r}},\bm{s})=d_{\bm{s}}(\overline{\bm{r}},\bm{s})_{f}=d_{\bm{s}}(\overline{f(\bm{r})},f(\bm{s}))$ for every $\bm{r}\in \mathbb{Z}^{N}$. Since $\bm{r},\bm{s}\in \mathbb{Z}^{N}\setminus\{\bm{0}\}$ are arbitrary, it follows that $ b_{\bm{r}}=d_{s}=c$ for every $\bm{r},\bm{s}\in \mathbb{Z}^{N}\setminus\{\bm{0}\}$.

We therefore have,
\[
\left(\overline{f(\bm{r})},f(\bm{s})\right)
= c(\overline{\bm{r}},\bm{s}) \qquad\forall\,\bm{r},\bm{s} \in \mathbb{Z}^{N}\setminus\{\bm{0}\}. \tag{7}\label{eq:7}
\]
Writing~\eqref{eq:7} in matrix form gives
\[
\bm{Q}^{\top}\bm{J\bm Q} = c\bm{J}.
\]
 Taking the determinant on both sides of the above equation, it follows that $c^{2m}=1$ and hence $c=\pm1$. Consequently we get $\bm{Q}^{\top}\bm{JQ}=\pm \bm{J}$, so $\bm{Q} \in \bm{GSp}_{N}(\mathbb{Z})$.
 
 From the definition of $c_{\bm{r}}$ we have
 \[
 \sigma(D(\overline{\bm{r}}, \bm{r}))=c_{\bm{r}}D\left(\overline{\bm{Q}\bm {r}},\bm{Q}\bm{r}\right)=c_{\bm{r}}D(\bm{JQ}\bm{r},\bm{Q}\bm{r}).
  \]
 Using $\bm{Q}^{T}\bm{JQ}=\pm \bm{J}$ we may rewrite the image as
 \[\sigma(D(\overline{\bm{r}}, \bm{r}))=\pm c_{\bm{r}}D\left(\bm{Q}^{-\top}\bm{J}\bm{r},\bm{Q}\bm{r}\right)=\pm c_{\bm{r}}D\left(\bm{Q}^{-\top}\overline{\bm{r}},\bm{Q}\bm{r}\right),  \tag{8}\label{eq:8}\] with the sign $\pm$ determined by whether $\bm{Q}^{\top}\bm{JQ}=\bm{J}$ or $\bm{Q}^{\top}\bm{J}\bm{Q}=-\bm{J}$.

 \medskip
 
 \textbf{Case 1:} For $\bm{Q}\in \mathbf{Sp}_{N}(\mathbb{Z}) $, we show that the scalars $c_{\bm{r}}$, for any $\bm{r}\in \mathbb{Z}^{N}\setminus\{\bm{0}\},$ will be of the form $c_{\bm{r}}= \lambda^{\bm{r}}$ for some  $\lambda\in (\mathbb{K^{\times}})^{N} $. More precisely, we have \begin{equation*}
     \sigma(D(\overline{\bm{r}},\bm{r}))=\lambda^{\bm{r}}D(\bm{Q}^{-\top}\overline{\bm{r}}, \bm{Q}\bm{r}).
 \end{equation*}
 
 \textbf{Claim 1:}     $c_{\bm{r}+\bm{s}}=c_{\bm{r}}c_{\bm{s}}$, if $(\overline{\bm{r}},\bm{s})\neq 0$.

 For $\bm{r},\bm{s}\in \mathbb{Z}^{N}\setminus\{\bm{0}\}$, we have
 \begin{equation*}
    \left[D(\overline{\bm{r}},\bm{r}), D(\overline{\bm{s}},\bm{s})\right]=(\overline{\bm{r}},\bm{s})D(\overline{\bm{r}}+\overline{\bm{s}}, \bm{r}+\bm{s} ).
 \end{equation*}
 Applying $\sigma$ on both sides of the above equation, we obtain
 \begin{equation*}
    c_{\bm{r}}c_{\bm{s}}\left(Q^{-\top} \overline{\bm{r}}, Q\bm{s}\right)D\left(Q^{-\top}(\overline{\bm{r}}+\overline{\bm{s}}),Q(\bm{r}+\bm{s})\right)=c_{\bm{r}+\bm{s}}(\overline{\bm{r}},\bm{s})D\left(Q^{-\top}(\overline{\bm{r}}+\overline{\bm{s}}),Q(\bm{r}+\bm{s})\right),
 \end{equation*}
 which gives $c_{\bm{r}+\bm{s}}=c_{\bm{r}}c_{\bm{s}}$. \\
\textbf{Claim 2:}
    $c_{-\bm{s}}=c_{\bm{s}}^{-1}$, for all $\bm{s}\in \mathbb{Z}^{N}\setminus\{\bm{0}\}$.

    For any $\bm{0}\neq \bm{s}\in \mathbb{Z}^{N}$, we choose $\bm{r}\in \mathbb{Z}^{N}\setminus \{\bm{0}\}$ with $(\overline{\bm{r}},\bm{s})\neq 0$, then $c_{\bm{r}}=c_{\bm{r}+\bm{s}-\bm{s}}=c_{\bm{r}+\bm{s}}c_{-\bm{s}}= c_{\bm{r}}c_{\bm{s}}c_{-\bm{s}}$ and hence $c_{-\bm{s}}=c_{\bm{s}}^{-1}$. \\
\textbf{Claim 3:}
    $c_{\bm{r}+\bm{s}}=c_{\bm{r}}c_{\bm{s}}$ for all $\bm{r},\bm{s}\in \mathbb{Z}^{N}\setminus\{\bm{0}\}$ with $\bm{r}+\bm{s}\neq \bm{0}$.
If $(\overline{\bm{r}},\bm{s})\neq 0$, the result follows from Claim 1. therefore it is sufficient to prove for the case $\bm{r},\bm{s}\in \mathbb{Z}^{N}\setminus\{\bm{0}\}$ with $(\overline{\bm{r}},\bm{s})=0$. For such $\bm{r}$ and $\bm{s}$, we choose $\bm{t}$ such that $(\overline{\bm{r}},\bm{t})\neq 0,\; (\overline{\bm{s}},\bm{t})\neq 0$ and $\left(\overline{\bm{r}+\bm{s}},\bm{t}\right)\neq 0$. Then we see that    $c_{\bm{r}+\bm{s}}=c_{\bm{r}+\bm{t}-\bm{t}+\bm{s}}= c_{\bm{r}+\bm{t}}c_{-\bm{t}+\bm{s}}=c_{\bm{r}}c_{\bm{t}}c_{-\bm{t}}c_{\bm{s}}=c_{\bm{r}}c_{\bm{s}}$.

\medskip

We set $c_{e_{i}}:=\lambda_{i}\in \mathbb{K}^{\times}$ for every $i\in \{1,\ldots,N\}$. For an arbitrary $\bm{r}=(r_1,\ldots ,r_N)^{\top}\in \mathbb{Z}^N\setminus\{\bm{0}\}$. We have $\bm{r}=\sum_{i=1}^{N}r_{i}e_{i}$, then using Claim 1 and Claim 2, we have $c_{r_{i}e_{i}}=c_{e_{i}}^{r_{i}}=\lambda_{i}^{r_{i}}$ for $r_{i}\neq 0$. Repeatedly using Claim~3 yields $c_{\bm{r}}=\lambda_{1}^{r_{1}}\ldots\lambda_{N}^{r_{N}}$.

\medskip

Finally, for all $\bm{r}\in \mathbb{Z}^{N}\setminus\{\bm{0}\}$, substituting the value of $c_{\bm{r}}$ in  (\ref{eq:8}), we obtain 
    \begin{align*}
        \sigma\left(D(\overline{\bm{r}},\bm{r})\right)= \lambda^{\bm{r}}D(\bm{Q}^{-\top}\overline{\bm{r}},\bm{Q}\bm{r})
    \end{align*}
    when $\bm{JQ} = \bm{Q}^{-\top}\bm{J}$. In terms of $h_{\bm{r}}$, this means \begin{equation}
        \sigma(h_{\bm{r}})=\lambda^{\bm{r}}h_{\bm{Q}\bm{r}}, \; \text{when}\; \bm{J}\bm{Q}=\bm{Q}^{-\top}\bm{J}.
    \end{equation}
    
    \medskip
    
    \textbf{Case 2:} When $\bm{Q}\in \mathbf{GSp}_{N}$ is such that $\bm{Q}^{\top}\bm{J}=-\bm{J}\bm{Q}$, then $c_{\bm{r}}$ equals $(-1)^{|\bm{r}|-1}\lambda^{\bm{r}}$, where $|\bm{r}|=\sum_{i=1}^{N}r_{i}$, that is
    \begin{equation*}
        \sigma(D(\overline{\bm{r}},\bm{r}))=(-1)^{|\bm{r}|}\lambda^{\bm{r}}D(\bm{Q}^{-\top}\overline{\bm{r}},\bm{Q}\bm{r}).
    \end{equation*}
    In other words, \[ \sigma(h_{\bm{r}})=(-1)^{|\bm{r}|-1}\lambda^{\bm{r}}h_{Q\bm{r}}, \; \text{for}\; \bm{J}\bm{Q}=-\bm{Q}^{-\top}\bm{J}.\]
     \textbf{Claim 4:}
         $c_{\bm{r}+\bm{s}}=-c_{\bm{r}}c_{\bm{s}}$ for $\bm{r},\bm{s}\in \mathbb{Z}^{N}\setminus\{\bm{0}\}$ with $(\overline{\bm{r}},\bm{s})\neq 0$. \\
     
     Following the same steps as in Claim 1 and noting that $(\bm{Q}^{-\top}\overline{\bm{r}},\bm{Q}\bm{s})=-(\overline{\bm{r}},\bm{s})$, it follows that $c_{\bm{r}+\bm{s}}=-c_{\bm{r}}c_{\bm{s}}$. \\
     \textbf{Claim 5:}
         $c_{-\bm{s}}=c_{\bm{s}}^{-1}$.\\
     
     As done in Claim 2, we find $\bm{r}$ such that $(\overline{\bm{r}},\bm{s})\neq 0$. Now using Claim 4, it follows that
     \begin{equation*}
        c_{\bm{r}}=c_{\bm{r}+\bm{s}-\bm{s}}=-c_{\bm{r}+\bm{s}}c_{-\bm{s}}=c_{\bm{r}}c_{\bm{s}}c_{-\bm{s}}.
    \end{equation*}
    The claim therefore follows.
    \textbf{Claim 6:}
        $c_{\bm{r}+\bm{s}}=c_{\bm{r}}c_{\bm{s}}$ for all $\bm{r},\bm{s}\in \mathbb{Z}^{N}\setminus\{\bm{0}\}$ with $\bm{r}+\bm{s}\neq \bm{0}$.\\
    
    We choose $\bm{t}\in \mathbb{Z}^{N}\setminus\{\bm{0}\}$ such that $(\overline{\bm{r}},\bm{t})\neq 0,\; (\overline{\bm{s}},\bm{t})\neq 0$ and $(\overline{\bm{r}+\bm{s}},\bm{t})\neq 0$. Using Claim 4 and Claim 5, we see that
    \begin{equation*}
        c_{\bm{r}+\bm{s}}=c_{\bm{r}+\bm{t}+\bm{s}-\bm{t}}=-c_{\bm{r}+\bm{t}}c_{\bm{s}-\bm{t}}=-c_{\bm{r}}c_{\bm{t}}c_{\bm{s}}c_{-\bm{t}}=-c_{\bm{r}}c_{\bm{s}}.
    \end{equation*}
    Now let $\lambda_{i}=c_{e_{i}} $ for all $1\leq i \leq N$. We note that $c_{ne_{i}}=(-1)^{n-1}\lambda_{i}^{n}$. As for $n\in \mathbb{Z}_{+}$, we see that 
    $c_{ne_{i}}=-c_{(n-1)e_{i}}\lambda_{i} = \cdots =(-1)^{n-1}\lambda_{i}^{n}$. \\
    Finally, if $\bm{r}=(r_{1},\ldots,r_{n})^{\top}\in \mathbb{Z}^{N}\setminus\{\bm{0}\}$, then we see that $c_{\bm{r}}=(-1)^{|\bm{r}|-1}\lambda^{\bm{r}}$, where $\lambda =(\lambda_{1},\ldots,\lambda_{N})^{\top}\in (\mathbb{K}^{\times})^{N}$ and $\lambda_{i}=c_{e_{i}}$. \\
    We write $\bm{r}=\sum_{i=1}^{N}r_{i}e_{i}$, then we have \begin{align*}
        c_{\bm{r}}=(-1)^{n-1}c_{r_{1}e_{1}}\cdots c_{r_{N}e_{N}} &=(-1)^{N-1}(-1)^{\sum_{i=1}^{N}(r_{i}-1)}\lambda_{1}^{r_{1}}\cdots \lambda_{N}^{r_{N}} \\
        &= (-1)^{|\bm{r}|-1}\lambda^{\bm{r}}
    \end{align*}
    Therefore, putting this expression for $c_{\bm{r}}$ in $\sigma(D(\overline{\bm{r}},\bm{r}))=-c_{\bm{r}}D(\bm{Q}^{-\top}\overline{\bm{r}},\bm{Q}\bm{r})$, we obtain
    \begin{equation*}
        \sigma(D(\overline{\bm{r}},\bm{r}))=(-1)^{|\bm{r}|}\lambda^{\bm{r}}D(Q^{-\top}\overline{\bm{r}}, \bm{Q}\bm{r}).
    \end{equation*}
\end{proof} 

We consider the lexicographic ordering on the set $\mathbb{Z}^{N}$ denoted by  $\prec$.
Note that any element of $\mathcal{H}_{N}'$ can be written as $\sum_{i=1}^{n}a_{i}h_{\bm{r}^{(i)}}$ with $\bm{r}^{(1)} \prec \bm{r}^{(2)} \prec \cdots \prec \bm{r}^{(n)} $. In the following context, we shall always assume such an expression without specifying it.

\begin{lemma}\label{lem:3.2}
    Let $\sigma \in \mathrm{Aut}(\mathcal{H}_{N}')$, then for every $\bm{r}\in \mathbb{Z}^{N}\setminus\{\bm{0}\}$ there exists $\bm{s}\in \mathbb{Z}^{N}\setminus\{\bm{0}\}$ such that $\sigma((\mathcal{H}_{N}')_{\bm{r}})=(\mathcal{H}_{N}')_{\bm{s}}$ and $\sigma(({\mathcal{H}_{N}'})_{-\bm{r}})=(\mathcal{H}_{N}')_{-\bm{s}}$.
\end{lemma}

\begin{proof}
We will prove by contradiction, proceeding in the following two steps.

\medskip

    \textbf{Step 1.} Suppose
    \begin{equation}\label{eq:9}
        \sigma(h_{\bm{r}}) = \sum_{i=1}^{m} c_{i}h_{\bm{r}^{(i)}}, \tag{9}
    \end{equation}
    \begin{equation}\label{eq:10}
        \sigma(h_{-\bm{r}}) = \sum_{i=1}^{n}d_{i}h_{\bm{s}^{(i)}}.\tag{10}
    \end{equation}
    where $c_{i}, d_{i}\in \mathbb{K}^{\times}$, $\bm{r}^{(i)},\bm{s}^{(i)}\in \mathbb{Z}^{N}\setminus\{\bm{0}\}$.  We first show that  $\bm{r}^{(1)}+ \bm{s}^{(1)} = \bm{0} $. Assume, for contradiction, that
     \[
     \bm{r}^{(1)}+\bm{s}^{(1)}\prec \bm{0}.
     \]
     Since $[h_{\bm{r}},h_{-\bm{r}}]=0$, applying $\sigma$ yields, 
     \begin{align*}
        [\sigma(h_{\bm{r}}),\sigma(h_{-\bm{r}})]&=0, \\
         \sum_{i=1}^{m}\sum_{j=1}^{n}c_{i}d_{j}[h_{\bm{r}^{(i)}},h_{\bm{s}^{(j)}}] &=0.
    \end{align*}
    Then the lowest degree term must vanish, i.e.
    \[
    c_{1}d_{1}[h_{\bm{r}^{(1)}},h_{\bm{s}^{(1)}}]=0.
    \]
    Thus
    \[
    \overline{\bm{r}^{(1)}},\overline{\bm{s}^{(1)}} \in \left(\bm{r}^{(1)}\right)^{\perp}\cap \left(\bm{s}^{(1)}\right)^{\perp}.
    \]
    We choose $\bm{t}\in \mathbb{Z}^{N}\setminus\{\bm{0}\}$ such that $\overline{\bm{r}^{(1)}},\overline{\bm{s}^{(1)}}\notin (\bm{t})^{\perp}$, then
    \begin{equation*}
      \left[\left[h_{\bm{t}},h_{\bm{r}^{(1)}}\right], h_{\bm{s}^{(1)}}\right]\neq 0.
    \end{equation*}
    For any $\sum_{k=1}^{l'}b_{k}h_{q^{(k)}}\in \mathcal{H}_{N}'$ with $q^{(1)}=\bm{t}$, there exists
    $\displaystyle\sum_{k=1}^{l}a_{k}h_{\rho^{(k)}}\in \mathcal{H}_{N}'$
    such that
    \[
    \sigma\left(\sum_{k=1}^{l}a_{k}h_{\rho^{(k)}}\right)= \sum_{k=1}^{l'}b_{k}h_{q^{(k)}}.
    \] \\
    Thus, for any $\bm{t}\in \mathbb{Z}^{N}\setminus\{\bm{0}\}$ with $\overline{\bm{r}^{(1)}},\overline{\bm{s}^{(1)}}\notin \bm{t}^{\perp}$ and $ \sum_{k=1}^{l'}b_{k}h_{q^{(k)}} \in \mathcal{H}_{N}'$ with $b_{1}\neq 0$ and $q^{(1)}=\bm{t}$, there exists $\sum_{k=1}^{l}a_{k}h_{\rho^{(k)}}\in \mathcal{H}_{N}'$ with $a_{k}\in \mathbb{K}^{\times}$ for every $k\in \{1,\ldots, l\} $ such that $\sigma(\sum_{k=1}^{l}a_{k}h_{\rho^{(k)}})=\sum_{k=1}^{l'}b_{k}h_{q^{(k)}}$. \\
    Consider the set 
   
   \[
\bigcup_{\substack{
\bm{t}\in \mathbb{Z}^{N}\setminus\{\bm{0}\} \\
\overline{\bm{r}^{(1)}},\,\overline{\bm{s}^{(1)}} \notin \bm{t}^{\perp}
}}
\left\{
l\in \mathbb{N}\;\middle|\;
\sigma\!\left(\sum_{k=1}^{l} a_{k}h_{\rho^{(k)}}\right)
=
\sum_{k=1}^{l'} b_{k}h_{q^{(k)}},
\text{ for some } a_k\neq 0,\; b_1\neq 0,\; q^{(1)}=\bm{t}
\right\}.
\]
   
    The above set is non-empty subset of $\mathbb{N}$, therefore this set has the least element. We will use the property of this least element in the Step 2 to arrive at a contradiction.

    \medskip 
    
    \textbf{Step 2.} Let $\displaystyle\sum_{k=1}^{p}a_{k}h_{\rho^{(k)}}, \; a_{k}\neq 0$  and $\displaystyle\sum_{k=1}^{q}b_{k}h_{q^{(k)}}, b_{1}\neq 0$ are such that \begin{equation}\label{eq:11}
        \sigma\left(\sum_{k=1}^{p}a_{k}h_{\rho^{(k)}}\right)=\sum_{k=1}^{q}b_{k}h_{q^{(k)}}\tag{11}
    \end{equation} 
    with $\overline{\bm{r}^{(1)}},\; \overline{\bm{s}^{(1)}} \notin (\bm{q}^{(1)})^{\perp}$ and 
    \begin{equation}\label{eq:12}
        \left[\left[h_{q^{(1)}},h_{\bm{r}^{(1)}}\right],h_{\bm{s}^{(1)}}\right]\neq 0.\tag{12} 
    \end{equation}
    The existence $\displaystyle\sum_{k=1}^{p}a_{k}h_{\rho^{(k)}}$ and $\displaystyle\sum_{k=1}^{q}b_{k}h_{q^{(k)}}$ is guaranteed by Step 1. Let $p$ be the least positive integer that satisfies the above condition. We have the following two cases:

    \medskip
    \noindent
    \textbf{Case 1:} If $p=1$, then we have
    \begin{equation*}
        \left[\left[h_{\rho^{(1)}},h_{\bm{r}}\right],h_{-\bm{r}}\right]=-\left(\rho^{(1)},\bm{r}\right)^{2}h_{\rho^{(1)}}.
    \end{equation*}
    Applying $\sigma $ on both sides of the above equation, we have
    \begin{align}
        \notag  \left[\left[\sigma\left(h_{\rho^{(1)}}\right), \sigma(h_{\bm{r}})\right],\sigma(h_{-\bm{r}})\right]&= -\left(\rho^{(1)},\bm{r}\right)^{2}\sigma\left(h_{\rho^{(1)}}\right), \\
        \left[\left[\sum_{k=1}^qb_{k}h_{q^{(k)}}, \sum_{i=1}^{m}c_{i}h_{\bm{r}^{(i)}}\right],\sum_{j=1}^{n}d_{j}h_{\bm{s}^{(j)}}\right]&=-\left(\rho^{(1)},\bm{r}\right)^{2}\sum_{k=1}^{q}b_{k}h_{q^{(k)}}.\tag{13}
    \end{align}
    Notice that the lowest degree term on the LHS has degree $\bm{r}^{(1)}+\bm{s}^{(1)}+q^{(1)} \prec q^{(1)}$, while the minimal term on RHS has degree $ q^{(1)}$. These degrees cannot match since the graded components intersect trivially
    \[
    (\mathcal{H}_{N}')_{q^{(1)}}\cap (\mathcal{H}_{N}')_{q^{(1)}+\bm{r}^{(1)}+\bm{s}^{(1)}}=\{0\},
    \]
    which contradicts~\eqref{eq:12}.

    \medskip
    \noindent
    \textbf{Case 2:} If $p>1$, we note that
    \begin{equation*}
        \left[\left[\sum_{k=1}^{p}a_{k}h_{\rho^{(k)}},h_{\bm{r}}\right],h_{-\bm{r}}\right]=-\sum_{k=1}^{p}a_{k}\left(\overline{\rho^{(k)}},\bm{r}\right)^{2}h_{\rho^{(k)}}.
    \end{equation*}
    Applying $\sigma$ to both sides of the above equation, we have
    \begin{equation*}
        \left[\left[\sum_{k=1}^{q}b_{k}h_{q^{(k)}},\sigma(h_{\bm{r}})\right], \sigma(h_{-\bm{r}})\right]=\sigma\left(-\sum_{k=1}^{p}a_{k}\left(\overline{\rho^{(k)}},\bm{r}\right)^{2}h_{\rho^{(k)}}\right).
    \end{equation*}
    Using~\eqref{eq:9} and~\eqref{eq:10}, we have
    \begin{equation}\label{eq:14}
        \sum_{k=1}^{q}\sum_{i=1}^{m}\sum_{j=1}^{n}b_{k}c_{i}d_{j}\left[\left[h_{q^{(k)}},h_{\bm{r}^{(i)}}\right],h_{\bm{s}^{(j)}}\right]=\sigma\left(-\sum_{k=1}^{p}a_{k}\left(\overline{\rho^{(k)}},\bm{r}\right)^{2}h_{\rho^{(k)}}\right).\tag{14}
    \end{equation}
    If $\left(\overline{\rho^{(k)}},\bm{r}\right)$ are same for every $k$,  then we get minimal degree terms of unequal degree in~\eqref{eq:14}, a contradiction to~\eqref{eq:12}.

    Otherwise, we have $k_{0}$ such that $\left(\overline{\rho^{(k_{0})}},\bm{r}\right)^{2}\neq 0$ and 
    \begin{align*}
        \sum_{k=1}^{p}-a_{k}\left(\overline{\rho^{(k)}},\bm{r}\right)^{2}h_{\rho^{(k)}} &=-a_{k_{0}}\left(\overline{\rho^{(k_{0})}},\bm{r}\right)^{2}h_{\rho^{(k_{0})}} + \sum_{k\neq k_{0}}-a_{k}\left(\overline{\rho^{(k)}},\bm{r}\right)^{2}h_{\rho^{(k)}}, \\
        &= \left(\overline{\rho^{(k_{0})}},\bm{r}\right)^{2}\sum_{k=1}^{p}a_{k}h_{\rho^{(k)}} + \sum_{k\neq k_{0}}\left(a_{k}\left(\overline{\rho^{(k_{0})}},\bm{r}\right)^{2}-a_{k}\left(\overline{\rho^{(k)}},\bm{r}\right)^{2}\right)h_{\rho^{(k)}}.
    \end{align*}
    Applying $\sigma$ on both sides of the above equation, we obtain the expression for RHS of~\eqref{eq:14} as:
    \begin{multline}\label{eq:15}
           \sigma\left(\sum_{k=1}^{p}-a_{k}\left(\overline{\rho^{(k)}},\bm{r}\right)^{2}h_{\rho^{(k)}}\right) \\= 
           \left(\overline{\rho^{(k_{0})}},\bm{r}\right)^{2}\sum_{k=1}^{p}b_{k}h_{q^{(k)}} + \sigma \left(\sum_{k\neq k_{0}}\left(a_{k}\left(\overline{\rho^{(k_{0})}},\bm{r}\right)^{2}-a_{k}\left(\overline{\rho^{(k)}},\bm{r}\right)^{2}\right)h_{\rho^{(k)}}\right).\tag{15}
    \end{multline}
    Since $\bm{r}^{(1)}+\bm{s}^{(1)} \prec \bm{0}$, it follows from the~\eqref{eq:14} and~\eqref{eq:15} that
    \[
    b_{1}c_{1}d_{1}\left[\left[h_{q^{(1)}},h_{\bm{r}^{(1)}}\right],h_{\bm{s}^{(1)}}\right]
    \]
    is the minimal term of $\displaystyle \sigma \left(\sum_{k\neq k_{0}}\left(a_{k}\left(\overline{\rho^{(k_{0})}},\bm{r}\right)^{2}-a_{k}\left(\overline{\rho^{(k)}},\bm{r}\right)^{2}\right)h_{\rho^{(k)}}\right)$. By the choice of $p$, it follows that
    \begin{equation}\label{eq:16}
        \left[\left[h_{\bm{r}^{(1)}+\bm{s}^{(1)}+q^{(1)}},h_{\bm{r}^{(1)}}\right],h_{\bm{s}^{(1)}}\right]=0.\tag{16}
    \end{equation}
Using~\eqref{eq:16} and the fact that $\overline{\bm{r}^{(1)}},\overline{\bm{s}^{(1)}} \in \left(\bm{r}^{(1)}\right)^{\perp}\cap \left(\bm{s}^{(1)}\right)^{\perp} $,  we deduce that \begin{equation*}
    \left( \overline{q^{(1)}},\bm{r}^{(1)} \right)\left(\overline{q^{(1)}},\bm{s}^{(1)}\right)h_{2\bm{r}^{(1)}+2\bm{s}^{(1)}+q^{(1)}}=0. \end{equation*} 
Since $\overline{\bm{r}^{(1)}},\overline{\bm{s}^{(1)}}\notin (q^{(1)})^{\perp}$, it follows that $\left( \overline{q^{(1)}},\bm{r}^{(1)} \right)\neq 0$ and  $\left(\overline{q^{(1)}},\bm{s}^{(1)}\right)\neq 0$ and therefore the last equation implies $h_{2\bm{r}^{(1)}+2\bm{s}^{(1)}+q^{(1)}}=0$, which is only possible only when $ 2\bm{r}^{(1)}+2\bm{s}^{(1)}+q^{(1)}=0$, which implies that
$\overline{\bm{r}^{(1)}} \in \left(q^{(1)}\right)^{\perp}$ and $\overline{\bm{s}^{(1)}} \in \left(q^{(1)}\right)^{\perp}$, which contradicts the choice of $q^{(1)}$.

\medskip
\noindent

Thus, the case $\bm{r}^{(1)}+\bm{s}^{(1)}\prec\bm{0}$ is not possible. Similarly, we can deduce that $\bm{r}^{(1)}+\bm{s}^{(1)}\succ \bm{0}$ is not possible. Therefore the only possibility is $\bm{r}^{(1)}+\bm{s}^{(1)}=\bm{0}$.

\medskip
\noindent
By the same argument applied to the highest-degree terms, one can deduces that $\bm{r}^{(m)}+\bm{s}^{(n)}=\bm{0}$.
If $\bm{r}^{(1)} \prec \bm{r}^{(m)}$, then $\bm{s}^{(1)}=-\bm{r}^{(1)}\succ-\bm{r}^{(m)}=\bm{s}^{(n)}$, which contradicts $\bm{s}^{(1)}\prec \bm{s}^{(n)}$. Therefore we have $\bm{r}^{(1)}=\bm{r}^{(m)}=-\bm{s}^{(n)}=-\bm{s}^{(1)}$ and hence the result follows.
\end{proof}
The method used in Lemma~\ref{lem:3.2} refers to Lemma~4.2 in \cite{TX09}. \\
Note that for each $\bm{Q}\in \mathbf{GSp}_{N}(\mathbb{Z})$, there exists a naturally defined automorphism $\sigma_{\bm{Q}}$ of $\mathcal{H}_{N}'$, as follows
\begin{equation}\label{eq:17}\tag{17}
    \sigma_{Q}(h_{\bm{r}})=
    \begin{cases}
       h_{\bm{Q}\bm{r}}= D\left(\bm{Q}^{-\top}\overline{\bm{r}},\bm{Q}\bm{r}\right),\; & \forall \;\bm{r}\in \mathbb{Z}^{N}\setminus\{\bm{0}\} \; \text{for}\;  \bm{J}\bm{Q}=\bm{Q}^{-\top}\bm{J} \\
       (-1)^{|\bm{r}|-1}h_{\bm{Q}\bm{r}}=(-1)^{|\bm{r}|}D\left(\bm{Q}^{-\top}\overline{\bm{r}},\bm{Q}\bm{r}\right),\; & \forall \; \bm{r}\in \mathbb{Z}^{N}\setminus\{\bm{0}\} \;  \text{for}\; \bm{J}\bm{Q}=-\bm{Q}^{-\top}\bm{J}
    \end{cases}
\end{equation}
    Similarly, for any $\lambda \in (\mathbb{K}^{\times})^{N}$, one can define the natural automorphism $\sigma_{\lambda}$ as \begin{equation*}
        \sigma(h_{\bm{r}}) =\lambda^{\bm{r}}h_{\bm{r}}.
    \end{equation*}
    \begin{lemma}\label{lem:3.3}
        For $\bm{Q}\in \mathbf{GSp}_{N}(\mathbb{Z})$ and $\lambda \in (\mathbb{K}^{\times})^{N}$, there is $\mu =\mu(\lambda, \bm{Q})\in (\mathbb{K}^{\times})^{N}$ such that $\mu^{\bm{r}}=\lambda^{\bm{Q}\bm{r}}$.
    \end{lemma}
    \begin{proof}
        Take $\mu_{i}=\lambda^{\bm{Q}e_{i}}\in (\mathbb{K}^{\times})^{N}$ for every $i\in \{1,\ldots,N\}$, then for $\mu=(\mu_{1},\ldots,\mu_{N})^{\top}\in (\mathbb{K}^{\times})^{N}$ we have $\mu^{\bm{r}}=\lambda^{\bm{Q}\bm{r}}$.
    \end{proof}
\begin{theorem}\label{thm:3.4}
    $\operatorname{Aut}(\mathcal{H}_{N}')\cong \bm{GSp}_{N}(\mathbb{Z})\ltimes (\mathbb{K}^{\times})^{N}$.
\end{theorem}
\begin{proof}
    For $\sigma\in \operatorname{Aut}(\mathcal{H}_{N}')$, Lemma~\ref{lem:3.1} and Lemma~\ref{lem:3.2} implies the existence of $\bm{Q}\in \mathbf{GSp}_{N}(\mathbb{Z})$ and $\lambda\in (\mathbb{K}^{\times})^{N}$ such that \begin{equation*}
        \sigma(h_{\bm{r}})= \lambda^{\bm{r}}h_{\bm{Q}\bm{r}}, \; \text{for} \; \bm{J}\bm{Q}=\bm{Q}^{-\top}J
    \end{equation*}
    or
    \begin{equation*}
        \sigma(h_{\bm{r}})=(-1)^{|\bm{r}|-1}\lambda^{\bm{r}} h_{\bm{Q}\bm{r}}, \; \text{for} \; \bm{J}\bm{Q}=-\bm{Q}^{-\top}J.
    \end{equation*}
    
    If $\sigma_{\bm{Q}}$ and $\sigma_{\lambda}$ are the automorphism as defined above, then $\sigma = \sigma_{\bm{Q}}\sigma_{\lambda}$. Note that $\{\sigma_{\bm{Q}}\mid \bm{Q}\in \mathbf{GSp}_{N}(\mathbb{Z})\}$ and $\{\sigma_{\lambda} \mid \lambda \in (\mathbb{K}^{\times})^{N}\}$ are the subgroups of $\operatorname{Aut}(\mathcal{H}_{N}') $, which are isomorphic to $\mathbf{GSp}_{N}(\mathbb{Z})$ and $ (\mathbb{K}^{\times})^{N}$, respectively. Therefore, we have $\operatorname{Aut}(\mathcal{H}_{N}')\cong \mathbf{GSp}_{N}(\mathbb{Z}).(\mathbb{K}^{\times})^{N}$.    
    \\
    Therefore, it remains to prove that $\{\sigma_{\lambda} \mid \lambda \in (\mathbb{K}^{\times})^{N} \}\trianglelefteq \operatorname{Aut}(\mathcal{H}_{N}')$. 
    For any $\theta \in \operatorname{Aut}(\mathcal{H}_{N}')$, given by $\theta(h_{\bm{r}})= \lambda^{\bm{r}}h_{\bm{Q}\bm{r}}$ arising from $\bm{Q}\in \mathbf{GSp}_{N}(\mathbb{Z})$ such that $\bm{J}\bm{Q}=\bm{Q}^{-\top}\bm{J} $ or by $\theta(h_{\bm{r}})= (-1)^{|\bm{r}|-1}\lambda^{\bm{r}}h_{\bm{Q}\bm{r}}$ arising from $\bm{Q}\in \mathbf{GSp}_{N}(\mathbb{Z})$ such that $ \bm{J}\bm{Q}=-\bm{Q}^{-\top}\bm{J}$. It is easy to verify that \begin{equation*}
        \theta^{-1}\sigma_{\lambda}\theta(h_{\bm{r}})=\lambda^{\bm{Q}\bm{r}}h_{\bm{r}}.
    \end{equation*}
    If $\mu_{i}=\lambda^{\bm{Q}e_{i}}\in (\mathbb{K}^{\times})^{N}$ and $\mu=\mu_{1}\cdots\mu_{N}$, then $\theta^{-1}\sigma_{\lambda}\theta=\sigma_{\mu}$.
\end{proof}
The following lemma shows that any automorphism of Hamiltonian Lie algebra $\mathcal{H}_{N}$ maps $\mathcal{H}_{N}'$ to $\mathcal{H}_{N}'$ and  $\mathfrak{h}$ to $\mathfrak{h}$. 
\begin{lemma}\label{lem:3.5}
    If $\sigma \in \operatorname{Aut}(\mathcal{H}_{N})$, then $\sigma(\mathcal{H}_{N}')= \mathcal{H}_{N}'$ and $ \sigma(\mathfrak{h})=\mathfrak{h}$.
\end{lemma}
\begin{proof}
   Applying $\sigma$ to both the sides of the identity  $[\mathcal{H}_{N},\mathcal{H}_{N}]=\mathcal{H}_{N}'$, we obtain $\sigma([\mathcal{H}_{N},\mathcal{H}_{N}])=\sigma(H_{N}')$. As $\sigma$ is automorphism and $\sigma(\mathcal{H}_{N})=\mathcal{H}_{N}$, it follows that $\sigma(\mathcal{H}_{N}')=[\sigma(\mathcal{H}_{N}),\sigma(\mathcal{H}_{N})]=[\mathcal{H}_{N},\mathcal{H}_{N}]=\mathcal{H}_{N}'$. \\
   Now we show that $\sigma(\mathfrak{h})=\mathfrak{h}$. If not, then for some $0\neq u \in \mathbb{K}^{N}$ there exists  $ 0\neq v\in \mathbb{K}^{N}$ and  integral vectors $\bm{r}^{(1)}\prec \bm{r}^{(2)} \prec \cdots \prec \bm{r}^{(m)}$ such that \begin{equation}\label{eq:18}\tag{18}
      \sigma(D(u,\bm{0}))= D(v,\bm{0}) + \sum_{i=1}^{m}c_{i}h_{\bm{r}^{(i)}},
   \end{equation}
   where $0\neq \sum_{i=1}^{m}c_{i}h_{\bm{r}^{(i)}}\in \mathcal{H}_{N}'$. Then either $\bm{r}^{(1)}\prec \bm{0}$ or $ \bm{r}^{(m)}\succ \bm{0} $. \\
   If $\bm{r}^{(1)}\prec \bm{0}$, choose $\bm{s}\in \mathbb{Z}^{N}\setminus\{\bm{0}\}$ such that $ \overline{\bm{r}^{(1)}}\notin \bm{s}^{\perp}$, also \begin{equation*}
       \left[h_{\bm{s}}, h_{\bm{r}^{(1)}}\right]=\left(\overline{\bm{s}},\bm{r}^{(1)}\right)h_{\bm{s}+\bm{r}^{(1)}}\neq 0.
   \end{equation*}
   Since $\sigma(\mathcal{H}_{N}')=\mathcal{H}_{N}'$, for any $\sum_{k=1}^{l'}b_{k}h_{q^{(k)}}\in \mathcal{H}_{N}'$ with $b_{1}\neq0$, $q^{(1)}=\bm{s}$, there is $\sum_{k=1}^{l}a_{k}h_{\rho^{(k)}}\in \mathcal{H}_{N}'$ with $a_{k}\neq 0$ for every $k\in \{1,\ldots,l\}$ such that $\sigma \left(\sum_{k=1}^{l}a_{k}h_{\rho^{(k)}}\right)= \sum_{k=1}^{l'}b_{k}h_{q^{(k)}}$.\\
   Suppose $p>0$ be the least integer for which
   \begin{equation}\label{eq:19}\tag{19}
      \sigma\left(\sum_{k=1}^{p}a_{k}h_{\rho^{(k)}}\right)=\sum_{k=1}^{q}b_{k}h_{q^{(k)}},
   \end{equation} 
   with $\overline{\bm{r}^{(1)}}\notin \left(q^{(1)}\right)^{\perp}$. Note that such a $p$ exists by a similar argument as used in Step~1 of Lemma~~\ref{lem:3.2}. \\
   If $p=1$,  we apply $\sigma$ to both sides of the equation $ \left[D(u,\bm{0}),a_{1}h_{\rho^{(1)}}\right]= \left(u,\rho^{(1)}\right)a_{1}h_{\rho^{(1)}}$; using~\eqref{eq:18} and~\eqref{eq:19}, yields
   \begin{equation*}
       \left[D(v,\bm{0})+\sum_{i=1}^{m}c_{i}h_{\bm{r}^{(i)}},\sum_{k=1}^{q}b_{k}h_{q^{(k)}}\right]= -\left(u,\rho^{(1)}\right)\sum_{k=1}^{q}b_{k}h_{q^{(k)}},
   \end{equation*}
   which is an absurd, as the least degree terms are of different degree, namely, the least degree term in left hand side of above equation is $\bm{r}^{(1)}+q^{(1)}$ since $\left(\overline{\bm{r}^{(1)}},q^{(1)}\right)\neq 0$, whereas that of the right hand side is $q^{(1)}$.\\
   If $p>1$,  we apply $ \sigma$ to the equation
   \begin{equation}\label{eq:20}\tag{20}
       \left[D(u,\bm{0}),\sum_{k=1}^{p}a_{k}h_{{\rho}^{(k)}}\right]=\sum_{k=1}^{p}\left(u,\rho^{(k)}\right)a_{k}h_{\rho^{(k)}},
   \end{equation}
   and using~\eqref{eq:18} and~\eqref{eq:19}, we obtain
   \begin{equation}\label{eq:21}\tag{21}
      \left[D(v,\bm{0})+\sum_{i=1}^{m}c_{i}h_{\bm{r}^{(i)}}, \sum_{k=1}^{q}b_{k}h_{q^{(k)}}\right]=\sigma\left(\sum_{k=1}^{p}\left(u,\rho^{(k)}\right)a_{k}h_{\rho^{(k)}}\right)
   \end{equation}
   As done in the second case of Lemma 3.2, not all of the $\left(u,\rho^{(k)}\right)$ can be the same; otherwise we would obtain an equation having least degree terms of different degrees in both sides. We can therefore find $k\in \{1,\ldots, p\}$ for which $\left(u,\rho^{(k_{0})}\right)\neq 0$ and we have
   \begin{align}\label{eq:22}\tag{22}
      \sigma\left(\sum_{k=1}^{p}\left(u,\rho^{(k)}\right)a_{k}h_{\rho^{(k)}}\right)= \sigma\left(\left(u,\rho^{(k_{0})}\right)\sum_{k=1}^{p}a_{k}h_{\rho^{(k)}} - \sum_{k\neq k_{0}}\left(u, \rho^{(k)}-\rho^{(k_{0})}\right)a_{k}h_{\rho^{(k)}}\right).
   \end{align}
   Using~\eqref{eq:21} and~\eqref{eq:22}, we see that $q^{(1)}+\bm{r}^{(1)}$ is the least degree term of $ \sigma\left(\displaystyle\sum_{k\neq k_{0}}\left(u,\rho^{(k)}-\rho^{(k_{0})}\right)a_{k}h_{\rho^{(k)}}\right)$. Minimality of $p$ implies that $\overline{\bm{r}^{(1)}} \in \left(\bm{r}^{(1)}+q^{(1)}\right)^{\perp}$ and hence $\overline{\bm{r}^{(1)}} \in \left(q^{(1)}\right)^{\perp}$, a contradiction.
\end{proof}
\begin{theorem}\label{thm:3.6}
    $\operatorname{Aut}(\mathcal{H}_{N})\cong \operatorname{Aut}(\mathcal{H}_{N}')\cong \mathbf{GSp}_{N}(\mathbb{Z})\ltimes (\mathbb{K}^{\times})^{N}$.
\end{theorem}
\begin{proof}
    Let $\sigma \in \operatorname{Aut}(\mathcal{H}_{N})$, then Lemma~\ref{lem:3.5} implies that $\sigma$ induces an automorphism $ \sigma\big|_{\mathcal{H}_{N}'}$ of $\mathcal{H}_{N}'$. Consequently, by Lemma~\ref{lem:3.1} and Lemma~\ref{lem:3.2}, there exists $\bm{Q}\in \mathbf{GSp}_{N}(\mathbb{Z}) $ such that \begin{equation*}
        \sigma(h_{\bm{r}})=
        \begin{cases}
           \lambda^{\bm{r}}h_{\bm{Q}\bm{r}}, & \text{if} \; \bm{J}\bm{Q}=\bm{Q}^{-\top}\bm{J} \\
           (-1)^{|\bm{r}|-1}\lambda^{\bm{r}} h_{\bm{Q}\bm{r}}, & \text{if} \; \bm{J}\bm{Q}=-\bm{Q}^{-\top}\bm{J} 
        \end{cases}.
    \end{equation*}
    We show that any such $\sigma$ extends uniquely to an automorphism of $\mathcal{H}_{N}$. By Lemma~\ref{lem:3.5}, $\sigma(\mathfrak{h})=\mathfrak{h}$. For $u\in \mathbb{K}^{N}\setminus\{\bm{0}\}$, let $\sigma(D(u,\bm{0}))=D(v,0)$, we show that $v\in \mathbb{K}^{N}\setminus\{\bm{0}\}$ is uniquely determined by $\sigma$. Note that
    \begin{equation*}
        [D(u,0),h_{\bm{r}}]=(u,\bm{r})h_{\bm{r}}.
    \end{equation*}
    Apply $\sigma$ to the above equation, we get
    \begin{align*}
        [D(v,0), h_{\bm{Q}\bm{r}}]&=(u,\bm{r})h_{\bm{Q}\bm{r}}   \\
        (v,\bm{Q}\bm{r})h_{\bm{Q}\bm{r}}&=(u,\bm{r})h_{\bm{Q}\bm{r}}.
        \end{align*}
     We choose a basis of $\mathbb{K}^{N}$ consisting of $\bm{r}_{1},\ldots,\bm{r}_{N}\in \mathbb{Z}^{N}\setminus\{\bm{0}\}$ with $\bm{r}_{i}\notin u^{\perp}=\{v\in \mathbb{K}^{N} \mid \left( v, u\right)=0\}$ for every $i\in \{1,\ldots, N\}$. From the last equation, it follows that $(v,\bm{Q}\bm{r}_{i})=(u,\bm{r}_{i})$, which is $ \left(\bm{Q}^{\top}v,\bm{r}_{i}\right)=(u,\bm{r}_{i})$, which gives $ \left(\bm{Q}^{\top}v-u,\bm{r}_{i}\right)=0$ for all $1\leq i \leq N$. Non-degeneracy of the bilinear form $\left(,\right)$ implies $ \bm{Q}^{\top}v=u$ and hence $ v= \bm{Q}^{-\top}u $. Hence
     \begin{equation*}
         \sigma(D(u,\bm{0}))=D\left(\bm{Q}^{-\top}u,\bm{0}\right), \; \forall \, u \in \mathbb{K}^{N},
     \end{equation*}
    which completes the proof. 
\end{proof}
\begin{remark}
    Note that for $N=2$, we have $\mathbf{GSp}_{2}(\mathbb{Z})=\mathbf{GL}_{2}(\mathbb{Z})$, as for any $A\in \mathbf{GL}_{2}(\mathbb{Z})$, it is easy to see that $\bm{Q}^{\top}\bm{J}\bm{Q}=\operatorname{det}(A)J$. Also we know that $\mathcal{S}_{2}=\mathcal{H}_{2}$. Therefore the group obtained for $\mathcal{H}_{2}'$ as well for  $\mathcal{H}_{2}$ are the same as that obtained in \cite{TX09}.
\end{remark}
\section{Derivations of Hamiltonian Lie algebra}

We first give a generating set for the Hamiltonian Lie algebra $\mathcal{H}_{N}'$ with the following Lemma.
\begin{lemma}\label{lem:4.1}
    $\mathcal{H}_{N}'=\langle h_{\pm e_{i}},\ h_{\pm e_{j}\pm e_{k}} \mid \forall 1\leq i \leq N, 1\leq i < j \leq N\rangle$. 
\end{lemma}
\begin{proof}
Follows from Proposition~\ref{prop:2.3}.
\end{proof}

Our next aim is to show that every derivation of $\mathcal{H}_{N}'$ is inner. For this purpose, we recall the following lemma.
\begin{lemma}\label{lem:4.2}{\textbf{[Proposition 1.1 of [F]]}}
    Suppose that $G$ is an additive group. If $\mathfrak{g}=\oplus_{\alpha \in G} \mathfrak{g}_{\alpha}$ is a finitely generated $G$-graded Lie algebra, then
    \begin{equation*}
        \text{Der}(\mathfrak{g}) =\oplus_{\alpha\in G}(\text{Der}(\mathfrak{g}))_{\alpha}
    \end{equation*}
    is also $G$-graded, and it satisfies $(\text{Der}(\mathfrak{g}))_{\alpha}(\mathfrak{g}_{\beta})\subset \mathfrak{g}_{\alpha+\beta}$, for all $\alpha, \beta \in G$.
\end{lemma}

By Lemma-\ref{lem:4.1}, we know that $ \mathcal{H}_{N}'$ is a finitely generated $\mathbb{Z}^{N}$-graded Lie algebra. Therefore Lemma~\ref{lem:4.1} and Lemma~\ref{lem:4.2} implies that 
\begin{equation*}
    \text{Der}\left(\mathcal{H}_{N}')=\oplus_{\bm{r}\in \mathbb{Z}^{N}}(\text{Der}(\mathcal{H}_{N}')\right)_{\bm{r}}
\end{equation*}
is a $\mathbb{Z}^{N}$-graded Lie algebra. We now prove the following important lemma in this direction.
\begin{lemma}\label{lem:4.3}
    For any $\bm{r}\in \mathbb{Z}^{N}\setminus\{\bm{0}\}$, we have $(\text{Der}(\mathcal{H}'))_{\bm{r}}=\operatorname{ad}\;\mathcal{H}'_{\bm{r}}$.
\end{lemma}
\begin{proof}
    Let $\bm{r}=(r_{1}, \ldots ,r_{N})^{\top}\in \mathbb{Z}^{N}\setminus\{\bm{0}\}$ and $\tau=\gcd(r_{1}, \ldots ,r_{N})$, then $ \varepsilon_{1}=\tau^{-1}\bm{r}$ is a primitive vector of $\mathbb{Z}^{N}$. By Lemma~\ref{lem:2.1}, the group $\bm{Sp}_{N}(\mathbb{Z})$ acts transitively on the set of primitive vectors, therefore there exists $\bm{Q}\in \bm{Sp}_{N}(\mathbb{Z}) $ such that  $\bm{Q}e_{1}=\varepsilon_{1}$. For $2\leq i \leq N$, we set $\bm{Q}e_{i} = \varepsilon_{i}$. By Theorem~\ref{thm:3.4}, the matrix $\bm{Q}$ determines an automorphism $\sigma$ of $\mathcal{H}_{N}'$ given by $\sigma(h_{\bm{r}})=\lambda^{\bm{r}}h_{\bm{Q}\bm{r}}$. As automorphism takes a generating set of a Lie algebra to another generating set, it follows that 
    \begin{equation*}
        \langle h_{\pm \varepsilon_{i}},\, h_{\pm \varepsilon_{j} \pm \epsilon_{k}} \mid  \forall\;  i\in \{1,\ldots , N\} \; \text{and} \; \forall  j < k \in \{1,\ldots, N\} \rangle =\mathcal{H}_{N}'.
    \end{equation*}
    Let $\partial \in \operatorname{Der}(\mathcal{H}_{N}')_{\bm{r}}$, since each nonzero graded space of  $\mathcal{H}_{N}'$ is one dimensional, there exists $c \in \mathbb{K}$ such that
    \begin{equation*}
        \partial(h_{\varepsilon_{m+1}})=c h_{\varepsilon_{m+1}+\tau\varepsilon_{1}}.
    \end{equation*}
    We observe that
    \begin{equation*}
        \left(\partial+\frac{c}{\tau}\operatorname{ad}(h_{\tau\varepsilon_{1}})\right)h_{\varepsilon_{m+1}}=0.
    \end{equation*}
    Accordingly, we define $\displaystyle\partial' =\partial + \frac{c}{\tau}\operatorname{ad}(h_{\tau \varepsilon_{1}}) $, with this choice $\partial' $ vanishes on $h_{\pm \varepsilon_{i}}$ for all $i\neq 1$. Now suppose that $c_{1i}^{\pm}\in \mathbb{K}$ are such that  $\partial'(h_{\pm\varepsilon_{i}})=c_{1i}^{\pm} h_{\pm \varepsilon_{i}+\tau\varepsilon_{1}}$, since \begin{equation*}
        [h_{\pm \varepsilon_{i}}, h_{\epsilon_{m+1}}]=0.
    \end{equation*}
    Applying $\partial'$ to both sides yields \begin{equation*}
        \left[\partial'(h_{\pm \varepsilon_{i}}), h_{\varepsilon_{m+1}}\right]+\left[h_{ \pm \varepsilon_{i}}, \partial'(h_{\varepsilon_{m+1}})\right]=0.
    \end{equation*}
    As $ \partial'(h_{\varepsilon_{m+1}})=0$, we obtain $-\tau c_{1i}^{\pm} h_{\tau\varepsilon_{1}\pm \varepsilon_{k}+\varepsilon_{m+1}}=0$ \\
    and consequently, $c_{1i}^{\pm}=0$ for all $ 2\leq i \leq N$. \\
\textbf{Claim 1:}
    For every $\bm{s}=(s_1,\ldots ,s_N)^{\top}\in (\mathbb{Z}^{*})^{N}$ with $\tau\neq -2s_{1}$, we have $\partial'(h_{\bm{Q}\bm{r}})=0$. \\

Following the argument we used in the proof of Claim~3 of Lemma~\ref{lem:2.2}, we obtain
\begin{equation*}
    \left \langle h_{\bm{Qs}} \mid \bm{s}\in (\mathbb{Z}^{*})^N \right \rangle = \mathcal{H}_{N}'.
\end{equation*}
    If $\bm{s}'=\bm{Q}\bm{s}$ for some $ \bm{s} \in (\mathbb{Z}^{*})^{N}$ , we have
\begin{align*}
   \left[\left[h_{\bm{s}'},h_{\varepsilon_{m+1}}\right],h_{-\varepsilon_{m+1}}\right] &=[[h_{\bm{Q}\bm{s}}, h_{\bm{Q}e_{m+1}}], h_{-\bm{Q}e_{m+1}},], \\
   &=(\bm{J}\bm{Q}\bm{s}, \bm{Q}e_{m+1})[h_{\bm{Q}(\bm{s}+e_{m+1})}, h_{-\bm{Q}e_{m+1}}], 
   \\ & = (\bm{J}\bm{Q}\bm{s},\bm{Q}e_{m+1})(\bm{J}\bm{Q}(\bm{s}+e_{m+1}),-\bm{Q}e_{m+1} )h_{\bm{Q}\bm{s}}, \\
   &=(\bm{J}\bm{s},e_{m+1})(\bm{J}(\bm{s}+e_{m+1}), -e_{m+1})h_{\bm{Q}\bm{s}}, \; \text{since} \; \bm{Q}\in \mathbf{Sp}_{N}(\mathbb{Z}) \\
   &= -s_{1}^{2}h_{\bm{Q}\bm{s}}
   \\
   &=-s_{1}^{2}h_{\bm{s}'}.
\end{align*}
    Applying $\partial'$ to both sides of the above equality and using Leibnitz identity yields
\begin{equation*}
    \left[[\partial(h_{\bm{s}'}),h_{\varepsilon_{m+1}}]h_{-\varepsilon_{m+1}}\right]=-s_{1}^{2}h_{\bm{s}'},
\end{equation*}
which implies
\begin{equation*}
    -(s_{1}+\tau)^{2}ch_{\bm{s}'+\tau \varepsilon_{1}}= -cs_{1}^{2}h_{\bm{s}'+\tau \varepsilon_{1}},
\end{equation*}
and hence $c\left((s_{1}+\tau)^{2}-s_{1}^{2}\right)=0$ or equivalently, $c\tau(2r_{1}+\tau)=0$. Since $\tau\neq 0$, we get $c=0$ if  $\tau\neq -2s_{1}$. \\
\textbf{Claim 2:}
    For every $\bm{s}=(s_{1},\ldots,s_{N})^{\top}\in (\mathbb{Z}^{*})^{N}$ satisfying $\tau=-2s_{1}$, we have $\partial'(h_{\bm{Q}\bm{s}})=0$.
\\
Let $\bm{r}',\bm{r}'' \in (\mathbb{Z}^{*})^{N}$ be such that $\bm{r}' +\bm{r}''=\bm{s}$ and $\left(\overline{\bm{r}'},\bm{r}''\right)\neq 0$. By Claim 1, $\partial'(h_{\bm{Q}\bm{r}'})=0$, $\partial'(h_{\bm{Q}\bm{r}''})=0$ and $[h_{\bm{Q}\bm{r}'},h_{\bm{Q}\bm{r}''}]=\left( \overline{\bm{r}'},\bm{r}''\right)h_{\bm{Q}\bm{s}}$. On applying $\partial'$ on both sides of the previous equation, we get \begin{align*}
    (\overline{\bm{r}'},\bm{r}'')\partial'(h_{\bm{Q}\bm{s}})&= [\partial'(h_{\bm{Q}\bm{r}'}),h_{\bm{Q}\bm{r}''}] +[h_{\bm{Q}\bm{r}'}, \partial'(h_{\bm{Q}\bm{r}''})], \\
    &= 0, \; \; \;  \text{by Claim 1} 
\end{align*}
Therefore, we get $\partial'(h_{\bm{Q}\bm{s}})=0$ for every $\bm{s}=(s_{1},\ldots,s_{N})^{\top}\in (\mathbb{Z}^{*})^{N}$ with $\tau=-2s_{1}$. \\
\textbf{Claim 3:}
    For every $\bm{s}\in \mathbb{Z}^{N}\setminus\{\bm{0}\}$, we have $\partial'(h_{\bm{Q}\bm{s}})=0$.
\\
By Claim~1 and Claim~2, we already know that $\partial'(h_{\bm{Q}\bm{s}})=0$, for every $\bm{r}\in (\mathbb{Z}^{*})^{N}$. Now let $\bm{s}\in \mathbb{Z}^{N}\setminus\{ \bm{0}\}$, then by Claim~3 of Proposition~\ref{prop:2.3}, we can find $\bm{s}',\bm{s}'' \in (\mathbb{Z}^{*})^{N} $ satisfying $\bm{s}=\bm{s}'+\bm{s}''$ and $\left(\overline{\bm{s}'},\bm{s}''\right)\neq 0$. Therefore
    \begin{align*}
        \left[h_{\bm{Q}\bm{s}'}, h_{\bm{Q}\bm{s}''}\right] &= \left(\overline{\bm{Q}\bm{s}'},\bm{Q}\bm{s}''\right)h_{\bm{Q}\bm{s}} \\
        &=\left(\overline{\bm{s}'},\bm{s}''\right)h_{\bm{Q}\bm{s}}.
     \end{align*} 
   Applying $\partial'$ to sides of above equation and using Claim 2, we conclude that   $ \partial'(h_{\bm{Q}\bm{s}})=0$, which completes the  proof.
\end{proof}
The following lemma characterizes the outer derivations of $\mathcal{H}_{N}'$.
\begin{lemma}\label{lem:4.4}
    If $\partial \in \text{Der}(\mathcal{H}_{N}')_{\bm{0}}$, then there exists $h=\sum_{k=1}^{N}c_{k}D(e_{k},0)\in \mathfrak{h}$ such that $\partial = \operatorname{ad}(h) $ on $\mathcal{H}_{N}'$.
\end{lemma}
\begin{proof}
    Let $\bm{r}\in \mathbb{Z}^{N}\setminus\{\bm{0}\}$, since $\partial\in \text{Der}(\mathcal{H}_{N}')_{0}$ it follows that $ \partial(h_{\bm{r}})=c_{\bm{r}}h_{\bm{r}}$ for some $c_{\bm{r}}\in \mathbb{K}$.  Our aim is to determine the relation of $c_{\bm{r}+\bm{s}}$ with $c_{\bm{r}}$ and $c_{\bm{s}}$. We first observe that for $\bm{r},\bm{s}\in \mathbb{Z}^{N}\setminus\{\bm{0}\}$, satisfying $(\overline{\bm{r}},\bm{s})\neq 0$, we have
    \begin{equation}\label{eq:23}
        c_{\bm{r}+\bm{s}}=c_{\bm{r}}+c_{\bm{s}}.\tag{23}
    \end{equation}
    To show this, we apply $\partial$ to the identity $ [h_{\bm{r}},h_{\bm{s}}]=(\overline{\bm{r}},\bm{s})h_{\bm{r}+\bm{s}}$ to obtain $\left(c_{\bm{r}}+c_{\bm{s}}\right)[h_{\bm{r}},h_{\bm{s}}]=(\overline{\bm{r}},\bm{s})c_{\bm{r}+\bm{s}}h_{\bm{r}+\bm{s}}$, which immediately yields $c_{\bm{r}+\bm{s}}=c_{\bm{r}}+c_{\bm{s}}$, whenever $(\overline{\bm{r}},\bm{s})\neq 0$.

    \medskip
    \noindent
    Next, we show that $c_{-\bm{r}}=-c_{\bm{r}}$ for every $\bm{r}\in \mathbb{Z}^{N}\setminus\{\bm{0}\}$. To this end, we choose $\bm{s}\in \mathbb{Z}^{N}\setminus\{\bm{0}\}$ such that $(\overline{\bm{r}},\bm{s})\neq 0$. Since $\left(\overline{\bm{r}+\bm{s}},-\bm{s}\right)\neq 0$,~\eqref{eq:23} impies that $c_{\bm{s}}=c_{\bm{r}+\bm{s}}+c_{-\bm{s}}$. On the other hand as $(\overline{\bm{r}},\bm{s})\neq 0$, another application of~\eqref{eq:23} gives
    \begin{align*}
        c_{\bm{s}}&=c_{\bm{r}+\bm{s}} +c_{-\bm{r}} \\
        &= c_{\bm{r}}+c_{\bm{s}} +c_{-\bm{r}},
    \end{align*}
    and hence $c_{-\bm{r}}=-c_{\bm{r}}$ for all $\bm{r}\in \mathbb{Z}^{N}\setminus\{\bm{0}\}$ as claimed.

     \medskip 
     
    We now establish that $c_{\bm{r}+\bm{s}}=c_{\bm{r}}+c_{\bm{s}}$ for all $\bm{r},\bm{s}\in \mathbb{Z}^{N}\setminus\{\bm{0}\}$. 
    Note that it is enough to consider the case where $\bm{r}+\bm{s}\neq 0$ and $(\overline{\bm{r}},\bm{s})= 0$. Choose $\bm{s}'\in \mathbb{Z}^{N}\setminus\{\bm{0}\}$ such that $(\overline{\bm{r}},\bm{s}')\neq 0,\; (\overline{\bm{s}},\bm{s}')\neq 0,\; \left(\overline{\bm{r}+\bm{s}},\bm{s}'\right)\neq 0$. Since $\bm{r}+\bm{s}=(\bm{r}-\bm{s}')+(\bm{s}'+\bm{s})$ and $\left(\overline{\bm{r}-\bm{s}'}, \bm{s}' +\bm{s}\right)=\left(\overline{\bm{r}},\bm{s}'\right)-\left(\overline{\bm{s}'},\bm{s}\right)=\left(\overline{\bm{r}+\bm{s}},\bm{s}'\right)\neq 0$, it follows from~\eqref{eq:23} that
    \begin{equation*}
        c_{\bm{r}+\bm{s}}=c_{\bm{r}-\bm{s}'}+c_{\bm{s}'+\bm{s}}.
    \end{equation*}
    Using again~\eqref{eq:23}, together with the identity  $c_{-\bm{r}}=-c_{\bm{r}}$, we compute
    \begin{align}
       \notag c_{\bm{r}+\bm{s}}&=c_{\bm{r}}+c_{-\bm{s}'}+c_{\bm{s}'}+c_{\bm{s}} \\
        &= c_{\bm{r}}+c_{\bm{s}}.\tag{24}\label{eq:24}
    \end{align}
    Finally, let $ c_{e_{i}}=c_{i}$ for every  $1\leq i \leq N$.  For any $\bm{r}=(r_{1},\ldots,r_{N})^{\top}\in \mathbb{Z}^{N}$, a repeated application of~\eqref{eq:24} yields $c_{\bm{r}}=\sum_{i=1}^{N}r_{i}c_{i}$. Consequently, $\displaystyle\partial(h_{\bm{r}})=\left(\sum_{i=1}^{N}r_{i}c_{i}\right)h_{r}$ for every $\bm{r}\in \mathbb{Z}^{N}\setminus\{\bm{0}\}$. If we set $h=\sum_{i=1}^{N}c_{i}D(e_{i},0)$, then it can easily be verified that $\partial = \operatorname{ad}(h)$ on $\mathcal{H}_{N}'$. This completes the proof.
\end{proof}
We have the following theorem that characterizes $\text{Der}(\mathcal{H}_{N}')$.
\begin{theorem}\label{thm:4.5}
    $\text{Der}(\mathcal{H}_{N}')\cong \mathcal{H}_{N}'\rtimes \mathfrak{h}\cong \mathcal{H}_{N}$.
\end{theorem}
\begin{proof}
    Using  Lemma ~\ref{lem:4.2}, Lemma~\ref{lem:4.3} and Lemma~\ref{lem:4.4}, it is easy to see that $ \mathrm{Der}(\mathcal{H}_{N}')=\operatorname{ad}(\mathcal{H}_{N}')\oplus \operatorname{ad}(\mathfrak{h})$. By Proposition~\ref{prop:2.4}, it is known that the Lie algebra $\mathcal{H}_{N}'$ is a simple Lie algebra, therefore, its center $Z(\mathcal{H}_{N}')=0$. Hence $ \operatorname{ad} : \mathcal{H}_{N}' \to \operatorname{ad}(\mathcal{H}_{N}')$ is a Lie algebra isomorphism. Also the map $\operatorname{ad}: \mathfrak{h} \to \operatorname{ad}(\mathfrak{h})  $ can be easily verified to be a Lie algebra isomorphism. Therefore, the result follows.
\end{proof}
\begin{theorem}\label{thm:4.6}
   $\text{Der}(\mathcal{H}_{N})\cong \mathcal{H}_{N}$.
\end{theorem}
\begin{proof}
By Proposition~\ref{prop:2.4}, $\mathcal{H}_{N}'$ is simple and therefore it is perfect and its center is trivial. By Theorem~\ref{thm:4.5}, we know that $\text{Der}(\mathcal{H}_{N}')=\mathcal{H}_{N}$. It follows from Theorem 1.1. of \cite{SZ05} that $ \mathcal{H}_{N}$ is complete i.e. $\text{Der}(\mathcal{H}_{N})\cong \mathcal{H}_{N}$.
\end{proof}
\section{Second Cohomology group of Hamiltonian algebra}
For any Lie algebra  $L$ over a field $\mathbb{K}$ of characteristic zero, we begin by recalling the definition of second cohomology group of a Lie algebra $L$ with trivial coefficients. For additional details, please refer to  Chapter IV of \cite{Knapp88}.
For any $n\in \mathbb{Z}_{+}$, define the space of $n$-cochains as  $C^{n}(L;\mathbb{K})=\operatorname{Hom}_{\mathbb{K}}(\Lambda^{n}L;\mathbb{K})$, which is the space of alternating $n$-linear forms on $L$. the coboundary map $d_{n}: C^{n}(L;\mathbb{K})\to C^{n+1}(L;\mathbb{K})$ is defined as \[(d_{n}f)(x_{1},\ldots,x_{n+1})=\sum_{1\leq r<s\leq n+1}(-1)^{r+s}f([x_{r},x_{s}], x_{1},\ldots, \hat{x_{r}},\ldots, \hat{x_{s}},\ldots,x_{n+1}), \]
\\
where $f\in C^{n}(L;\mathbb{K})$ and $x_{1},\ldots,x_{n+1}\in L$. \\
The coboundary operators satisfies \begin{equation}\label{eq:25}d_{n}d_{n-1}=0. \tag{25} \end{equation}
We define the space $Z^{n}(L;\mathbb{K})$ of $n$-cocycles as $\operatorname{ker}d_{n}$ and the space $B^{n}(L;\mathbb{K})$ of $n$-coboundaries as $\operatorname{image}d_{n-1}$. \\
Using~\eqref{eq:25}, we have $\operatorname{im}d_{n-1}\subseteq \operatorname{ker}d_{n}$. The space $H^{n}(L;\mathbb{K})=\operatorname{ker}d_{n}/\operatorname{im}d_{n-1}$ is the $n$th cohomology group of $L$ with trivial coefficients. \par   We are interested in finding the second cohomology group $H^{2}(L;\mathbb{K})$ for $L=\mathcal{H}_{N}$. For $n=2$, we first write the defining conditions of cochains, cocycles and coboundaries explicitly. \\
By a 2-cochain on $L$ we mean a $\mathbb{K}$-bilinear function $\psi:L\times L \to \mathbb{K}$ which satisfies \begin{equation*}
    \psi(x_{1},x_{2})=-\psi(x_{2},x_{1}), \forall \; x_{1},x_{2}\in L.
\end{equation*} 
By a 2-cocycle on $L$ we mean a $\mathbb{K}$-bilinear function $\psi: L\times L \to \mathbb{K}$ satisfying \begin{equation}\label{eq:26}
    \psi(x_{1},x_{2})=-\psi(x_{2},x_{1}), \; \forall \; x_{1},x_{2}\in L,\tag{26}
\end{equation}
and 
\begin{equation}\label{eq:27}
    \psi([x_{1},x_{2}],x_{3})+\psi([x_{2},x_{3}],x_{1})+\psi([x_{3},x_{1}],x_{2})=0, \; \forall \; x_{1},x_{2},x_{3}\in L.\tag{27}
\end{equation}
 \\
By a 2-coboundary on $L$ we mean a 2-cocycle $\psi$ which can be expressed \begin{equation}\label{eq:28}
    \psi=d_{1}f, \tag{28}
\end{equation} 
for some $1$-cochain $f$ on $L$. \\
For any $\mathbb{Z}^{N}$-graded Lie algebra $L$ and $\psi\in C^{n}(L;\mathbb{K})$, we define a homogeneous $n$-cochain of degree $\bm{k}$ as \begin{equation}\label{eq:29}
    \psi_{k}(x_{1},\ldots,x_{n})=\psi(x_{1},\ldots,x_{n})\delta_{\bm{r}_{1}+\cdots+\bm{r}_{n},\bm{k}}, \tag{29} \end{equation}
where $x_{i}\in L_{\bm{r}_{i}},\forall \; i\in \{1,\ldots ,n\}$ and the vector space of all cochains which can be expressed like~\eqref{eq:29}  is denote by $C^{n}(L;\mathbb{K})_{\bm{k}}$, whose elements are termed as homogeneous $n$-cochains of degrre $\bm{k}$.  We have the following regarding the coboundary operator $d_{n}$
\begin{lemma}\label{lem:5.1}
    The coboundary operator preserves the degree of homogeneous cochains of a $\mathbb{Z}^{N}$-graded Lie algebra, that is, $d_{n}(C^{n}(L;\mathbb{K})_{\bm{r}})\subseteq C^{n+1}(L;\mathbb{K})_{\bm{r}}$.
\end{lemma}
\begin{proof}
For any  $\psi\in C^{n}(L;\mathbb{K})_{\bm{k}}$
, $\psi(x_{1},\ldots,x_{n})\neq 0$ with $x\in L_{\bm{k}_{i}},\forall \; i\in \{1,\ldots,n\}$ implies $\sum_{i=1}^{n}\bm{k}_{i}=\bm{k}$, then $d_{n}\psi\in C^{n+1}(L,\mathbb{K})$ can be expressed as 
\begin{equation}\label{eq;30}
    d_{n}\psi(x_{1},\ldots,x_{n+1})=\sum_{1\leq r<s\leq n+1}(-1)^{r+s}\psi([x_{r},x_{s}], x_{1},\ldots, \hat{x_{r}},\ldots, \hat{x_{s}},\ldots,x_{n+1}),\tag{30}
\end{equation}
 where $x_{i}\in L_{\bm{k}_{i}},\forall \; i\in \{1,\ldots,n+1\}$. One right hand side, we note that for any $1\leq r<s\leq n+1$, $[x_{r},x_{s}]\in L_{\bm{k}_{r}+\bm{k}_{s}}$. Hence in any term of right hand side, the sum of degree of components of $\psi$ is $\sum_{i=1}^{n+1}\bm{k}_{i}$. Therefore by definition definition of $\psi$, \[ \psi([x_{r},x_{s}],, x_{1},\ldots, \hat{x_{r}},\ldots, \hat{x_{s}},\ldots,x_{n+1}) \neq 0\] implies $\sum_{i=1}^{n+1}r_{i}=\bm{k}$. Thus, $d_{n}\psi\in C^{n+1}(L;\mathbb{K})_{\bm{k}}$.
\end{proof}
  We first prove the following important lemma for any $\mathbb{Z}^{N}$-graded Lie algebra $L$

\begin{lemma}\label{lem:5.2}
    For any $\mathbb{Z}^{N}$-graded Lie algebra $L$, the $2$-cohomology group can be expressed as
    
\begin{equation}\label{eq:31}
    H^{2}(L;\mathbb{K})= \prod_{\bm{r}\in \mathbb{Z}^{N}}H^{2}(L; \mathbb{K})_{\bm{r}}, \tag{31}
\end{equation}
where $H^{2}(L; \mathbb{K})_{\bm{r}}= Z^{2}(L; \mathbb{K})_{\bm{r}}/B^{2}(L;\mathbb{K})_{\bm{r}}, \forall \; \bm{r}\in \mathbb{Z}^{N}$.
\end{lemma}
\begin{proof} First we see that any $2$-cochain can be uniquely written as a  sum of homogeneous $2$-cochains on $L=\mathcal{H}_{N}$. Let $\psi\in C^{2}(L;\mathbb{K})$ be a 2-cochain on $L$, we define a collection of $ \psi_{\bm{k}}$ for every $\bm{k}\in \mathbb{Z}^{N}$. For $x\in L_{\bm{r}}$ and $y\in L_{\bm{s}}$, we define \[\psi_{\bm{k}}(x,y)=\begin{cases} \psi(x,y), \; \text{if}\; \bm{r}+\bm{s}=\bm{k}\\
0, \; \; \; \; \; \; \; \; \; \;\text{if}\; \bm{r}+\bm{s}\neq \bm{k}
    
\end{cases}\]
\\
i.e. $\psi_{\bm{k}}(x,y)=\psi(x,y)\delta_{\bm{r}+\bm{s},\bm{k}}$ for every $x\in L_{\bm{r}},y \in L_{\bm{s}}$. Now for any $x,y\in L$, we write $x=\sum_{\bm{r}}x_{\bm{r}}$ and $y=\sum_{s}y_{\bm{s}}$. Bilinearlity of $\psi$ gives \begin{align*}
    \psi(x,y)=\psi(\sum_{\bm{r}}x_{\bm{r}},\sum_{\bm{s}}y_{\bm{s}})&= \sum_{\bm{r}}\sum_{\bm{s}}\psi(x_{\bm{r}},y_{\bm{s}}) \\
    &=\sum_{\bm{k}}  \sum_{\substack{\bm{r},\bm{s}\\ \bm{r}+\bm{s}=\bm{k}}}\psi(x_{\bm{r}},y_{\bm{s}}) \\ 
    &=\sum_{\bm{k}}\sum_{\substack{\bm{r},\bm{s}\\ \bm{r}+\bm{s}=\bm{k}}}\psi_{\bm{k}}(x,y)
\end{align*}

This means that each element $\psi\in C^{2}(L; \mathbb{K})$ can be written as the formal sum of homogeneous $2$-cocycles i.e. $\psi=\sum_{\bm{r}\in \mathbb{Z}^{N}}\psi_{\bm{r}}$. Next we show that $\sum_{\bm{r}}\psi_{\bm{r}}=0$ implies $\psi_{\bm{r}}=0$. For a given index $\bm{k}$, we choose $x\in L_{\bm{r}}$, $y \in L_{\bm{s}}$ such that $\bm{r}+\bm{s}=\bm{k}$, then $\sum_{\bm{r}}\psi_{\bm{r}}(x,y)=0$ implies $\psi_{\bm{k}}(x,y)=0$. Since $\bm{r},\bm{s}\in \mathbb{Z}^{N}$ are artibrary such that $\bm{r}+\bm{s}=\bm{k}$, it follows that $\psi_{\bm{k}}=0, \forall \; \bm{k}\in \mathbb{Z}^{N}$.
Thus, we have
\[C^{2}(L;\mathbb{K})=\prod_{\bm{r}\in \mathbb{Z}^{N}}C^{2}(L;\mathbb{K})_{\bm{r}},\] 
 where $C^{2}(L;\mathbb{K})_{\bm{r}}$ is the vector space of homogeneous 2-cocycles of degree $\bm{r}$. \\
Similarly, we have the following expressions for $1$-cochains and $3$-chains : \begin{align*}
    C^{1}(L;\mathbb{K})=\prod_{\bm{r}\in \mathbb{Z}^{N}}C^{1}(L;\mathbb{K})_{\bm{r}}, \\
 C^{3}(L;\mathbb{K})=  \prod_{\bm{r}\in  \mathbb{Z}^{N}}C^{3}(L;\mathbb{K})_{\bm{r}}
\end{align*} For the expression of $C^{1}(L;\mathbb{K})$, for any $1$-cochain $f$ and any  $ \bm{k}\in \mathbb{Z}^{N}$ we define homogeneous 1-cochain on $\mathcal{H}_{N}$ by \[ f_{\bm{k}}(x)=f(x_{\bm{k}}), \]  
where $x=\sum_{\bm{r}}x_{\bm{r}}$, with $x_{\bm{r}}\in L_{\bm{r}}$. Then $f$ is the formal sum of homogeneous $1$-cochains $f_{\bm{r}}$ i.e. $f=\sum_{\bm{r}\in \mathbb{Z}^{N}}f_{\bm{r}}$. \\
Similarly, any $3$-cochain $\psi$ can be written as a formal sum of homogeneous $3$-cochains $\psi_{\bm{k}}$ defined  as \[\psi_{k}(x_{1},x_{2},x_{3})=\psi(x_{1},x_{2},x_{3})\delta_{\bm{r}+\bm{s}+\bm{t},\bm{k}},\] where $x_{1}\in L_{\bm{r}}, x_{2}\in L_{\bm{s}}$ and $x_{3}\in L_{\bm{t}}$. Then $\psi=\sum_{\bm{r}\in \mathbb{Z}^{N}}\psi_{\bm{r}}$. \\
 Now we show that the expression of $C^{2}(L;\mathbb{K})$ as the direct product of the graded components $C^{2}(L;\mathbb{K})_{\bm{r}}$ induces similar expression for the spaces $Z^{2}(L;\mathbb{K})$ and $B^{2}(L;\mathbb{K})$. That is \begin{align} \label{eq:32}
        Z^{2}(L;\mathbb{K})&=\prod_{\bm{r}\in \mathbb{Z}^{N}}Z^{2}(L;\mathbb{K})_{\bm{r}},\tag{32} \\
        \label{eq:33}
         B^{2}(L;\mathbb{K})&=\prod_{\bm{r}\in \mathbb{Z}^{N}}B^{2}(L;\mathbb{K})_{\bm{r}},\tag{33}
    \end{align} \\
    \\
    For the expression of  $Z^{2}(L;\mathbb{K})$, we recall the differential $d:C^{k}(L;\mathbb{K})\to C^{k+1}(L;\mathbb{K})$ preserves the degree of homogeneous cochains. that a $2$-cochain $\psi$ is cocycle if $d\psi=0$. If $\psi=\sum_{\bm{k}\in \mathbb{Z}^{N}}\psi_{\bm{k}}$, the linearity of $d$ implies $d\psi =\sum_{\bm{r}\in\mathbb{Z}^{N}}d\psi_{\bm{r}}=0$. Lemma~\ref{lem:5.1} implies $d\psi_{\bm{r}} \in C^{3}(L;\mathbb{K})_{\bm{r}}$, and hence $\sum_{\bm{r}\in \mathbb{Z}^{N}}d\psi_{\bm{r}}=0$ implies $d\psi_{\bm{r}}=0, \forall \bm{r}$. Thus $\psi_{\bm{r}}$ is a $2$-cocycle for every $\bm{r} \in \mathbb{Z}^{N}$. Thus, we have  \begin{equation*}
        Z^{2}(L;\mathbb{K})=\prod_{\bm{r}\in \mathbb{Z}^{N}}Z^{2}(L;\mathbb{K})\cap C^{2}(L;\mathbb{K})_{\bm{r}}=\prod_{\bm{r}\in \mathbb{Z}^{N}}Z^{2}(L;\mathbb{K})_{\bm{r}}  \end{equation*} 
        For the expression of $B^{2}(L;\mathbb{K})$, we recall that a 2-cochain $\psi$ is a coboundary if there exists a $1$-cochain $f$ such that $df=\psi$. For the coboundary $\psi$, let $df=\psi$, then we write $f$ as a sum of homogeneous $1$-cochains i.e. $f=\sum_{\bm{r}\in \mathbb{Z}^{N}}f_{\bm{r}}$. Applying $d$ and using the linearity of $f$ yields $\psi =df=\sum_{\bm{r}\in \mathbb{Z}^{N}}df_{\bm{r}}$, where each $df_{\bm{r}}$ is homogeneous $2$-coboundary of degree $\bm{r}$. We just proved that \begin{equation*}
            B^{2}(L;\mathbb{K})=\prod_{\bm{r}\in \mathbb{Z}^{N}}B^{2}(L;\mathbb{K})\cap C^{2}(L;\mathbb{\mathbb{K}})_{\bm{r}}=\prod_{\bm{r}\in \mathbb{Z}^{N}}B^{2}(L;\mathbb{K})_{\bm{r}}.
        \end{equation*}
        Using the expression for $Z^{2}(L;\mathbb{K})$ and $B^{2}(L;\mathbb{K})$, we obtain~\eqref{eq:31}.
    \end{proof}
Using~\eqref{eq:31}, it is enough to consider homogeneous $2$-cocycles and 2-coboundaries of every possible degree $\bm{r}$. We now express homogeneous $2$-cocycles on $\mathcal{H}_{N}$ of degree $0$ .
\begin{proposition}\label{prop:5.3}
    Every homogeneous $2$-cocycle of degree $\bm{0}$ on $\mathcal{H}_{N}$ is of the form $\psi(h_{\bm{r}},h_{\bm{s}})=f(\bm{r})\delta_{\bm{r}+\bm{s},\bm{0}},$ $ \forall \bm{r},\bm{s}\in \mathbb{Z}^{N}$ and $\psi(d_{i},d_{j})=c_{ij}\in \mathbb{K},\; \forall i,j\in \{1,\ldots,N\}$, where $f(\bm{r})= (u,\bm{r})$ for some $u\in \mathbb{K}^{N}$ and $\bm{C}=[c_{ij}] $ is skew symmetric matrix of order $N$. 
\end{proposition}
\begin{proof}
    By definition, we have $\psi(h_{\bm{r}},h_{\bm{s}})=g(\bm{r},\bm{s})\delta_{\bm{r}+\bm{s},\bm{0}}$ for some function $g : \mathbb{Z}^{N}\times \mathbb{Z}^{N}\to \mathbb{K}$. Therefore, it is enough to determine $f(\bm{r})=g(\bm{r},-\bm{r})$ for every $\bm{r}\in \mathbb{Z}^{N}$. By condition~\eqref{eq:26}, $f$ is an odd function i.e.  $f(\bm{r})=-f(-\bm{r}), \forall \; \bm{r}\in \mathbb{Z}^{N}$. Putting $x=h_{\bm{r}}, y=h_{\bm{s}}$ and $h_{\bm{t}}$ in~\eqref{eq:27} yields \begin{align*}
         \left(\overline{\bm{r}},\bm{s} \right)\psi(h_{\bm{r}+\bm{s}},h_{\bm{t}}) + \left(\overline{\bm{s}},\bm{t} \right)\psi(h_{\bm{s}+\bm{t}},h_{\bm{r}}) + \left(\overline{\bm{t}},\bm{r} \right)\psi(h_{\bm{t}+\bm{r}},h_{\bm{s}})&=0, \\
        \left(\overline{\bm{r}},\bm{s} \right)g(\bm{r}+\bm{s},\bm{t})\delta_{\bm{r}+\bm{s}+\bm{t},\bm{0}} + \left(\overline{\bm{s}},\bm{t} \right)g(\bm{s}+\bm{t},\bm{r})\delta_{\bm{r}+\bm{s}+\bm{t},\bm{0}} +(\overline{\bm{t}},\bm{r})g(\bm{t}+\bm{r},\bm{s})\delta_{\bm{r}+\bm{s}+\bm{t},\bm{0}}&=0.
    \end{align*}
    When $\bm{r}+\bm{s}+\bm{t}=0$, the above expression gives 
    \begin{align*}
        \left(\overline{\bm{r}},\bm{s} \right)f(\bm{r}+\bm{s})+ \left(\overline{\bm{s}},-\bm{r}-\bm{s}\right)f(-\bm{r})+ \left( \overline{-\bm{r}-\bm{s}},\bm{r} \right)f(-\bm{s})=0
    \end{align*}
    Since $f$ is an odd function, the above equation becomes \begin{equation*}
        \left(\overline{\bm{r}},\bm{s} \right)f(\bm{r}+\bm{s})= \left(\overline{\bm{r}},\bm{s} \right)f(\bm{r})+\left(\overline{\bm{r}},\bm{s}\right)f(\bm{s}).
    \end{equation*}
    For $\left(\overline{\bm{r}},\bm{s} \right)\neq 0$, we have \begin{equation}\label{eq:34}
        f(\bm{r}+\bm{s})=f(\bm{r})+f(\bm{s}),\tag{34}
    \end{equation}
    Now we have a similar situation like that in Lemma~\ref{lem:4.4}. By the same argument as used in Lemma~\ref{lem:4.4}, we deduce that \begin{equation*}
        f(\bm{r}+\bm{s})= f(\bm{r})+f(\bm{s})
    \end{equation*}
    for every $\bm{r},\bm{s}\in \mathbb{Z}^{N}$. Consequently, if $f(e_{i})=u_{i}\in \mathbb{K}$ for $i\in \{1,\ldots,N\}$, then it follows that $f(\bm{r})=(u,\bm{r})$, where $u=(u_{1},\ldots, u_{N})^{\top}\in \mathbb{K}^{N}$. \par
    Next, we focus on the values of $\psi(d_{i},d_{j})$ for $i,j\in \{1,\ldots,N\}$. Let $\psi(d_{i},d_{j})= c_{ij}\in \mathbb{K}$ for $i,j\in \{1,\ldots,N\}$, then condition~\eqref{eq:26} implies that $ c_{ij}=-c_{ji}, \; \forall \; i,j \in \{1,\ldots,N\}$. However, since $[d_{i},d_{j}]= 0$ for $i,j\in \{1,\ldots,N\}$, the condition~\eqref{eq:27} imposes no restrictions on the constants $c_{ij}$. Therefore the matrix $\bm{C}=[c_{ij}]$ is a skew symmetric matrix over $\mathbb{K}$.
\end{proof}
We define the following two homogeneous 2-cocycles of degree $\bm{0}$ on $\mathcal{H}_{N}$ as $\psi_{1}(d_{i},d_{j})=c_{ij}$ and $\psi_{1}(h_{\bm{r}},h_{\bm{s}})=0,\forall \; \bm{r},\bm{s}\in \mathbb{Z}^{N}\setminus\{\bm{0}\}$, where $c_{ij}=-c_{ji},\forall i,j\in \{1,\ldots,N\} $. Also, define $\psi_{2}(d_{i},d_{j})=0, \forall \; i,j\in \{1,\ldots, N\}$ and $\psi_{2}(h_{\bm{r}},h_{\bm{s}})=(u,\bm{r})\delta_{\bm{r}+\bm{s},\bm{0}}$ for some $u \in \mathbb{K}^{N}$. \\
It follows by Proposition~\ref{prop:5.3} that any homogeneous $2$-cocycle $\psi$ of degree $\bm{0}$  on $\mathcal{H}_{N}$ can be uniquely written as $\psi= \psi_{1}+\psi_{2}$. \par  We define the following subspaces of $Z^{2}(L;\mathbb{K})$ as 
\begin{equation*}
    V_{1}= \{ \psi_{1}\mid \psi_{1}(d_{i},d_{j})=c_{ij}, \;  \forall i,j\in \{1,\ldots,N\}, \psi(h_{\bm{r}},h_{\bm{s}})= 0, \forall \bm{r},\bm{s}\in \mathbb{Z}^{N}\setminus\{\bm{0}\}, [c_{ij}]\in \mathfrak{so}_N\}
\end{equation*}
 and 
 \begin{equation*}
     V_{2}= \{\psi_{2}\mid \psi_{2}(d_{i},d_{j})=0, \psi_{2}(h_{\bm{r}},h_{\bm{s}})=(u,\bm{r})\delta_{\bm{r}+\bm{s},\bm{0}}, u \in \mathbb{K}^{N}\}
 \end{equation*}
Denote by $\Lambda^{2}(\mathbb{K}^{N})$ the second exterior power of $\mathbb{K}^{N}$. We have the following lemma.
\begin{lemma}\label{lem:5.4}
$Z^{2}(\mathcal{H}_{N};\mathbb{K})_{\bm{0}}=V_{1}\oplus V_{2}$, where $V_{1}\cong \Lambda^{2}(\mathbb{K}^{N})$ and $V_{2}\cong \mathbb{K}^{N}$ as additive groups.
\end{lemma}
\begin{lemma}\label{lem:5.5}
    The 2-cocycles $\psi_{1}$ and $\psi_{2}$ defined above are not coboundaries unless they are identically zero.
\end{lemma}
\begin{proof}
    Suppose there is a 1-cochain(a linear functional) on $\mathcal{H}_{N}$ such that $df=\psi_{1}$, that is, $df(x,y)=-f([x,y])= \psi_{1}(x,y)$ for $x,y\in \mathcal{H}_{N}$. In particular, we have $-f([d_{i},d_{j}])=\psi(d_{i},d_{j})=c_{ij},\; \forall \; i,j\in \{1,\ldots,N\}$, but $[d_{i},d_{j}]=0$ implies that $c_{ij}=0, \; i,j\in \{1,\ldots, N\}$ and hence $\psi_{1}\equiv0$.\\ 
    Similarly if $df=\psi_{2}$ for some 1-cochain on $\mathcal{H}_{N}$, then consider $0=-f([h_{\bm{r}},h_{-\bm{r}}])=\psi(h_{\bm{r}},h_{-\bm{r}})=(u,\bm{r})$, which means that $(u,\bm{r})=0, \forall \; \bm{r}\in \mathbb{Z}^{N}\setminus\{\bm{0}\}$ and hence $u=\bm{0}$ i.e. $\psi_{2}=0$.
\end{proof}
Using Lemma~\ref{lem:5.2}, Lemma~\ref{lem:5.4} and Lemma~\ref{lem:5.5}, the following corollary is immediate 
\begin{corollary}\label{cor:5.6}
    $H^{2}(\mathcal{H}_{N}; \mathbb{K})_{\bm{0}}\cong \Lambda^{2}(\mathbb{K}^{N})\oplus \mathbb{K}^{N}$.
\end{corollary}
Next we show that any homogeneous 2-cocycle of degree $\bm{k}\neq \bm{0}$ in $\mathcal{H}_{N}$ is actually a coboundary, In the following Lemma, we denote the coboundary operator $d_{1}: C^{1}(\mathcal{H}_{N};\mathbb{K})\to C^{2}(\mathcal{H}_{N}; \mathbb{K})$  by symbol $d$ to avoid confusion with the zero degree derivation $d_{1}\equiv t_{1}\frac{\partial}{\partial t_{1}}$.
\begin{lemma}\label{lem:5.7}
    $H^{2}(\mathcal{H}_{N};\mathbb{K})_{\bm{k}}=0$ for every $\bm{k}\in \mathbb{Z}^{N}\setminus\{\bm{0}\}$
\end{lemma}
\begin{proof}
    Let $\psi$ be a homogeneous 2-cocycle of degree $\bm{k}\neq \bm{0}$ in $\mathcal{H}_{N}$. Let $k_{i}\neq 0$ and for any $\bm{r},\bm{s}\in \mathbb{Z}^{N}\setminus\{\bm{0}\}$ with $\bm{r}+\bm{s}=\bm{k}$, we note that \begin{align}\label{eq:35}\tag{35}
        \psi([d_{i},h_{\bm{r}}],h_{\bm{s}})+\psi([h_{\bm{r}},h_{\bm{s}}],d_{i})+\psi([h_{s},d_{i}],h_{\bm{r}})&=0,\\\notag
        r_{i}\psi(h_{\bm{r}},h_{\bm{s}})+(\overline{\bm{r}},\bm{s})\psi(h_{\bm{r}+\bm{s}},d_{i})-s_{i}\psi(h_{\bm{s}},h_{\bm{r}})&=0.
    \end{align}
    Using skew symmetry of $\psi$, the above equation  yields 
    \begin{equation*}
        \psi(h_{\bm{r}},h_{\bm{s}})=\frac{\left(\overline{\bm{r}},\bm{s}\right)}{k_{i}}\psi(d_{i},h_{\bm{r}+\bm{s}}),
    \end{equation*}
    for any $\bm{r},\bm{s}\in \mathbb{Z}^{N}\setminus\{\bm{0}\}$ with $\bm{r}+\bm{s}=\bm{k}$. \\
    Define a linear functional $f$ on $\mathcal{H}_{N}$ by $f(h_{\bm{m}})= -\frac{1}{k_{i}}\psi(d_{i},h_{m})\delta_{\bm{k},\bm{m}}$ for every $\bm{m}\in \mathbb{Z}^{N}\setminus\{\bm{0}\}$ and $f(d_{j})=0$ for every $j\in \{1,\ldots,N\}$. \\
    We note that $df(h_{\bm{r}},h_{\bm{s}})=\psi(h_{\bm{r}}, h_{\bm{s}})$ for $\bm{r},\bm{s}\in \mathbb{Z}^{N}\setminus\{\bm{0}\}$ with $\bm{r}+\bm{s}=\bm{k}$ and $df(d_{j},h_{\bm{k}})=\psi(d_{i},h_{\bm{k}})$. \\
    We have $df(h_{\bm{r}},h_{\bm{s}})=-f([h_{\bm{r}},h_{\bm{s}}])=-\left(\overline{\bm{r}},\bm{s}\right)f(h_{\bm{r}+\bm{s}})=\frac{\left( \overline{\bm{r}},\bm{s}\right)}{k_{i}}\psi(d_{i},h_{\bm{r}+\bm{s}})$ for every $\bm{r},\bm{s}\in \mathbb{Z}^{N}\setminus\{\bm{0}\}$ with $\bm{r}+\bm{s}=\bm{k}$. \;\
    If $k_{j}=0$ for some $j\in \{1,\ldots, N\}$, then we note that $\psi(d_{j},h_{\bm{k}})=0$. We choose $\bm{r},\bm{s}\in \mathbb{Z}^{N}\setminus\{\bm{0}\}$ such that $\bm{r}+\bm{s}=\bm{k}$ and $\left(\overline{\bm{r}},\bm{s}\right)\neq 0$, then we use condition~\eqref{eq:27} on 2-cocycle to obtain  
    \\
    \begin{align*}
        r_{j}\psi(h_{\bm{r}},h_{\bm{s}})+\left(\overline{\bm{r}},\bm{s} \right)\psi(h_{\bm{k}},d_{i})-s_{j}\psi(h_{\bm{s}},h_{\bm{r}})&=0, \\
        (r_{j}+s_{j})\psi(h_{\bm{r}},h_{\bm{s}})+\left(\overline{\bm{r}},\bm{s}\right)\psi(h_{\bm{k}},d_{i})&=0.
        \end{align*}
        Since $r_{j}+s_{j}=k_{j}=0$ and $\left(\overline{\bm{r}},\bm{s}\right)\neq 0$, it follows that $\psi(d_{i},h_{\bm{k}})=0$. Consequently, we have \begin{equation*}
            0=-f([d_{j},h_{\bm{k}}])=\psi(d_{j},h_{\bm{k}}).
        \end{equation*}
        Finally, it remains to show that $-f([d_{j},h_{\bm{k}}])=\psi(d_{j},h_{k})$, where $k_{j}\neq 0$. For $i=j$, we are done by the definition of $f$. For $j\neq i$, we repeat the same process as done in~\eqref{eq:35} to obtain \begin{equation*}
            \psi(h_{\bm{r}},h_{\bm{s}})=\frac{\left(\overline{\bm{r}},\bm{s} \right)}{k_{j}}\psi(d_{j},h_{\bm{r}+\bm{s}}), \; \forall \bm{r},\bm{s}\in \mathbb{Z}^{N}\setminus\{\bm{0}\}, \bm{r}+\bm{s}=\bm{k}.
        \end{equation*}
        Thus for any $\bm{r},\bm{s}\in \mathbb{Z}^{N}\setminus\{\bm{0}\}$ with $\left(\overline{\bm{r}},\bm{s}\right)\neq 0$, the above equation gives \begin{equation*}
            \frac{1}{k_{j}}\psi(d_{j},h_{\bm{k}})=\frac{1}{k_{i}}\psi(d_{i},h_{\bm{k}})
        \end{equation*}
        Therefore, $-f([d_{j},h_{\bm{k}}])=-k_{j}f(h_{k})=k_{j}\frac{1}{k_{i}}\psi(d_{i},\bm{k})$. Finally, employing the last equation gives $df(d_{j},h_{\bm{k}})= -f([d_{j},h_{\bm{k}}])=\psi(d_{j},h_{\bm{k}})$.
\end{proof}
\begin{theorem}\label{thm:5.8}
$H^{2}(\mathcal{H}_{N};\mathbb{K})\cong \Lambda^{2}(\mathbb{K}^{N})\oplus \mathbb{K}^{N}$.
\end{theorem}
\begin{proof}
    The proof follows from Lemma~\ref{lem:5.2}, Corollary~\ref{cor:5.6} and Lemma~\ref{lem:5.7}. \end{proof}
%\bibliographystyle{amsplain} 
%\bibliography{your_bib_file_name} 

\begin{thebibliography}{100}
\bibitem[Ba16]{Ba16} V. V. Bavula, \emph{The Groups of automorphisms of the Witt $\mathcal{W}_N$ and Virasoro Lie algebras}, Czechoslovak Mathematical Journal, 66 (141) (2016), 1129–1141. 

\bibitem[Ba17]{Ba17} V. V. Bavula, \emph{The group of automorphisms of the Lie algebra of derivations of a polynomial algebra.} Algebra Appl. 16 (2017), 175–183. DOI: http://dx.doi.org/10.1142/
S0219498817500888.

\bibitem[Ben25]{}  Yves Benoist,  \emph{On the Rational Symplectic Group}. In: Pevzner, M., Sekiguchi, H. (eds) Symmetry in Geometry and Analysis, Volume 1. Progress in Mathematics, vol 357. Birkhäuser, Singapore.

\bibitem[BK25]{BK25} O. Bezushchak, I. Kashuba \emph{Derivations of Lie algebras of vector fields in infinitely many variables} arXiv:2507.04541.

\bibitem[DZ98]{DZ98} D. Z. Dokovic, K. Zhao \emph{Derivations, Isomorphisms, and second cohomology of generalized Witt algebras,} Trans. of Am. Math. Soc. 350 (1998).
\bibitem[Knapp88]{Knapp88} A.~W.~Knapp, \textit{Lie Groups, Lie Algebras, and Cohomology}, Vol.~34, Princeton University Press, 1988.
\url{https://doi.org/10.2307/j.ctv18zhdw5}.
\bibitem[R23]{R23}
Eswara Rao, S. \emph{Hamiltonian extended affine Lie algebra and its representation theory.} J. Algebra 628, 71–97 (2023).
\bibitem[Rao]{Rao} Eswara Rao S \emph{Private communication.}
\bibitem[F]{} R. Farnsteiner, Derivations and central extensions of finitely generated graded Lie algebras, J. Algebra 118 (1988) 33–45. 
\bibitem[FT25]{FT25}
V.~Futorny and S.~Tantubay,
\emph{Representations of Hamiltonian Lie algebras},
arXiv:2503.00749, 2025.
\bibitem[MT11]{MT11}  Malle G, Testerman D. \emph{Linear Algebraic Groups and Finite Groups of Lie Type.} Cambridge University Press; 2011. 

\bibitem[R74]{R74} \emph{On derivations of Lie algebras} Compositio Mathematica, tome 28, no 3 (1974), p. 299-303. 

\bibitem[R86]{R86} A. N. Rudakov, \emph{Subalgebras and automorphisms of Lie algebras of Cartan types,} Funct. Anal. App. 20 (1986) 72-73.
\bibitem[RP25]{RP25}
S.~E. Rao and S.~Pal,
\emph{Representations of Hamiltonian vector fields on a torus},
arXiv:2504.14517, 2025.
 \bibitem[SZ05]{SZ05} Y. Su, L. Zhu, Derivation algebras of centerless perfect Lie algebras are complete, J. Algebra 285 (2005) 508–515.


 \bibitem[TX09]{TX09} Xiao-Min Tang, Jin-Li Xu,
The derivation Lie algebra of the higher rank Virasoro-like algebra and its automorphism groups,
Linear Algebra and its Applications,
Volume 430, Issues 8–9,
2009,
Pages 2170-2181,
ISSN 0024-3795,
https://doi.org/10.1016/j.laa.2008.11.019.




 \end{thebibliography}

\end{document}